\documentclass[oneside,a4paper,11pt,reqno]{amsart}
\usepackage{color}
\usepackage{amssymb,amsmath,amsthm}
\usepackage{graphics,graphicx}
\usepackage{dsfont}
\usepackage[latin1]{inputenc}
\usepackage[T1]{fontenc}

\newtheorem{hypo}{Hypothesis}

\newtheorem{prop}[hypo]{Proposition}

\newtheorem{thm}[hypo]{Theorem}

\newtheorem{lem}[hypo]{Lemma}

\newtheorem{defi}[hypo]{Definition}

\newtheorem{rqe}[hypo]{Remark}

\newtheorem{coro}[hypo]{Corollary}

\newtheorem{exa}[hypo]{Example}

\DeclareMathOperator{\base}{Base}

\DeclareMathOperator{\val}{val}
\DeclareMathOperator{\sing}{Sing}

\def\F{\mathcal{F}}

\newcommand {\refeq}[1] {(\ref{#1})}

\title{Normal class and normal lines of algebraic hypersurfaces}
\date\today

\author{Alfrederic Josse}
\address{Universit\'e de Brest,
Laboratoire de Math\'ematiques de Bretagne Atlantique, UMR CNRS 6205, 29238 Brest cedex, France}
\email{alfrederic.josse@univ-brest.fr}

\author{Fran\c{c}oise P\`ene}
\address{Universit\'e de Brest and Institut Universitaire de France,
Laboratoire de Math\'ematiques de Bretagne Atlantique, UMR CNRS 6205, 29238 Brest cedex, France}
\email{francoise.pene@univ-brest.fr}

\subjclass[2000]{14J99,14H50,14E05,14N05,14N10}
\keywords{projective normal, class, Pl\"ucker, Grassmannian, Puiseux.}

\begin{document}
\begin{abstract}
We are interested in the normal class of an algebraic hypersurface $\mathcal Z$ in
the complexified euclidean projective space $\mathbb P^n$, that is the
number of normal lines to $\mathcal Z$ passing through a generic point of $\mathbb P^n$.
Thanks to the notion of normal polars, we state a formula for the normal class valid for a general hypersurface $\mathcal Z\subset\mathbb P^n$. We give a generic result and illustrate our formula on examples
in $\mathbb P^n$. We define the orthogonal incidence variety and
compute the Schubert class of the variety of projective normal lines to a surface of $\mathbb P^3$ in the Chow ring of $\mathbb G(1,3)$.
We complete our work with a generalization of Salmon's formula for the normal class of a Pl\"ucker
curve to any plane curve with any kind of singularity.
\end{abstract}
\maketitle
\section*{Introduction}
The notion of normal lines to an hypersurface of an euclidean space
is extended here to the complexified euclidean projective space $\mathbb P^n$ ($n\ge 2$). 
In this setting, $\mathcal H^\infty$ the hyperplane at infinity is fixed, together with
the umbilical at infinity $\mathcal U_\infty\subset\mathcal H^\infty$, the smooth quadric in $\mathcal H^\infty$ corresponding to the intersection
of $\mathcal H^\infty$ with any hypersphere (see Section \ref{DEFI0} for details). 
The aim of the present work is the study the {\bf normal class} $c_\nu(\mathcal Z)$ of a hypersurface $\mathcal Z$ of 
$\mathbb P^n$,
that is the number of $m\in\mathcal Z$ such that the projective normal line $\mathcal N_{m}(\mathcal Z)$ to 
$\mathcal Z$ at $m$ passing through a generic $m_1\in\mathbb P^n$
(see Section \ref{SEC00} for details). 
Our estimates provide upper bounds for the number of normal lines, of a real algebraic surface
in an $n$-dimensional affine euclidean space $E_n$, passing through a generic point in $E_n$.
Let us consider the \textbf{variety $\mathfrak{N}_{\mathcal{Z}}$ of projective normal lines of }$ \mathcal Z$ by
\[
\mathfrak{N}_{\mathcal{Z}}:=\overline{\{\mathcal{N}_{m}(\mathcal{Z});m\in \mathcal{Z}\}}\subset
 \mathbb{G}(1,n)\subset \mathbb P^{\frac{n(n+1)}2-1}
\] 
and its Schubert class $\mathfrak{n}_{\mathcal{Z}}:=[\mathfrak{N}_{\mathcal{Z}
}]\in A^{n-1}(\mathbb{G}(1,n))$ (when $\dim \mathfrak{N}_{\mathcal{Z}}=n-1$).
The fact that $PGL(n,\mathbb C)$ {\bf does not preserve normal lines} complicates our study compared to the study of tangent hyperplanes.
We prove namely the following result valid for a wide family of surfaces 
of $\mathbb P^n$.
Let $\mathcal Z=V(F)$ be an irreducible hypersurface of $\mathbb P^n$. We write $\mathcal Z_\infty:=\mathcal Z\cap
\mathcal H^\infty$.
Note that the singular points of $\mathcal Z_\infty$ correspond
to the points of tangency of $\mathcal Z$ with $\mathcal H^\infty$.

\begin{thm}\label{thmhypersurface}
Let $\mathcal Z\in\mathbb P^n$ be a smooth irreducible hypersurface of degree $d_{\mathcal Z}\ge 2$ such that $\mathcal H^\infty$ is not tangent to $\mathcal Z$ and that
at any $m\in\mathcal Z_\infty\cap\mathcal U_\infty$, the tangent planes to 
$\mathcal Z_\infty$ and to $\mathcal U_\infty$ at $m$ are distinct.
Then the normal class $c_\nu(\mathcal Z)$ of $\mathcal Z$ is
$$c_\nu(\mathcal Z)=d_{\mathcal Z}\sum_{k=0}^{n-1}(d_{\mathcal Z}-1)^k.$$
In particular,
\begin{itemize}
\item if $d_\mathcal Z=2$, $c_\nu(\mathcal Z)=n$;
\item if $n=2$, $c_\nu(\mathcal Z)=d_{\mathcal Z}$;
\item if $n=3$, $c_\nu(\mathcal Z)=d_{\mathcal Z}^3-d_{\mathcal Z}^2+d_{\mathcal Z}$;
\item if $n=4$, $c_\nu(\mathcal Z)=d_{\mathcal Z}^4-2d_{\mathcal Z}^3+2d_{\mathcal Z}^2$;
\item if $n=5$, $c_\nu(\mathcal Z)=d_{\mathcal Z}^5-3d_{\mathcal Z}^4+4d_{\mathcal Z}^3-2d_{\mathcal Z}^2+d_{\mathcal Z}$.
\end{itemize}
The normal class of an hyperplane $\mathcal H\subset\mathbb P^n$ (other than $\mathcal H^\infty$) is $c_\nu(\mathcal H)=1$.
\end{thm}
Actually we establish a general formula which is valid for a wider family of hypersurfaces of $\mathbb P^n$.
The notion of normal polars $\mathcal P_{A,\mathcal Z}$
plays an important role in our study. It is a notion analogous to the notion
of polars \cite{Dolga}.
Given an irreducible hypersurface 
$\mathcal Z\subset\mathbb P^n$ of degree $d_{\mathcal Z}$, we extend the definition of the line 
$\mathcal N_m(\mathcal S)$ to any $m\in\mathbb P^n$. 
We then define a regular map $\alpha_{\mathcal Z}:\mathbb P^n\setminus\mathcal B^{(0)}_{\mathcal Z}\rightarrow
\mathbb P^{\frac{n(n+1)}2-1}$ corresponding to $m\mapsto\mathcal N_m(\mathcal Z)$ (where $\mathcal B^{(0)}_\mathcal Z$
is the set of base points of $\alpha_{\mathcal Z}$).
We will see that $\mathcal B_{\mathcal Z}:=\mathcal B^{(0)}_{\mathcal Z}\cap\mathcal Z$ corresponds to the union of the set of singular points of $\mathcal Z$, of the set of points of tangency of $\mathcal Z$ with $\mathcal H^\infty$ and of the set of points of tangency of $\mathcal Z_\infty$ with $\mathcal U_\infty$.
For any $A\in\mathbb P^n$, we will introduce the notion of {\bf normal polar}
$\mathcal P_{A,\mathcal Z}$ of $\mathcal Z$ with respect to $A$
as the set of $m\in\mathbb P^n$
such that either  $m\in\mathcal B^{(0)}_{\mathcal Z}$ or $A\in \mathcal N_m(\mathcal Z)$.
We will see that, if $\dim\mathcal B^{(0)}_{\mathcal Z}\le 1$, then, for a generic $A\in\mathbb P^n$, 
$$\dim\mathcal P_{A,\mathcal Z}=1\quad\mbox{and}\quad
\deg \left(\mathcal P_{A,\mathcal Z}\right)= \sum_{k=0}^{n-1}(d_{\mathcal Z}-1)^k.$$
\begin{thm}\label{formulegeneralehypersurface}
Let $\mathcal Z$ be an irreducible hypersurface of $\mathbb P^n$
with isolated singularities, admitting a finite number of
points of tangency with $\mathcal H^\infty$ and such that 
$\mathcal Z_\infty$
has a finite number of  points of tangency with $\mathcal U_\infty$.
Then the normal class $c_\nu(\mathcal Z)$ 
of $\mathcal Z$ is given by
\[
c_\nu(\mathcal Z)=d_{\mathcal Z}.\sum_{k=0}^{n-1}(d_{\mathcal Z}-1)^k 
-\sum_{P\in B_{\mathcal Z}} i_P(\mathcal Z,{\mathcal P}_{A,\mathcal Z})\, ,
\]
for a generic $A\in \mathbb P^n$, where $i_P(\mathcal Z,{\mathcal P}_{A,\mathcal Z})$
is the intersection multiplicity of $\mathcal Z$ with ${\mathcal P}_{A,\mathcal Z}$.
\end{thm}
In dimension 3, we obtain the following result.
\begin{thm}[n=3, normal class and Chow ring]\label{formulegeneralesurface}
Let $\mathcal S$ be an irreducible surface of $\mathbb P^3$
with isolated singularities, admitting a finite number of
points of tangency with $\mathcal H^\infty$ and such that 
$\mathcal S_\infty$
has a finite number of  (non singular) points of tangency with $\mathcal U_\infty$.
Then 
$$\mathfrak{n}_{\mathcal{S}}=c_\nu(\mathcal S).\sigma_2+d_{\mathcal S}(d_{\mathcal S}-1).\sigma _{1,1}\in A^{2}(\mathbb{G}(1,3)),$$
where the normal class $c_\nu(\mathcal S)$ 
of $\mathcal Z$ is equal to 
$d_{\mathcal Z}.\deg(\mathcal P_{A,\mathcal Z}) $ (for a generic $A\in \mathbb P^n$) 
minus the sum of the intersection multiplicities of $\mathcal S$
with its generic normal polars ${\mathcal P}_{A,\mathcal S}$ at points of $\mathcal B_{\mathcal S}$.
\end{thm}
\begin{coro}[n=3]
For a generic irreducible surface $\mathcal S\subset \mathbb P^3$ of degree $d\ge 2$, we have
$c_\nu(\mathcal S)=d^3-d^2+d$ and
\[
{\mathfrak n}_{\mathcal{S}}=(d^3-d^2+d).\sigma_2+d(d-1).\sigma _{1,1}\in A^{2}(\mathbb{G}(1,3)).
\]
\end{coro}
In the next statement, we consider smooth surfaces $\mathcal S$ of $\mathbb P^3$
($\mathbb P^3$ being endowed with projective coordinates $[x:y:z:t]$)
such that $\mathcal S_\infty$
has no worse singularities than ordinary multiple points and 
ordinary cusps.
\begin{thm}[n=3]\label{thmsurfaces}
Let $\mathcal S\subset\mathbb P^3$ be a smooth irreducible
surface of degree $d_{\mathcal S}\ge 2$ such that:
\begin{itemize}
\item[(i)] in $\mathcal H^\infty$, the curve $\mathcal S_\infty$ has a finite number
of points of tangency with $\mathcal U_\infty$,
\item[(ii)]  any singular point of $\mathcal S_\infty$ is either an 
ordinary multiple point or an ordinary cusp,
\item[(iii)] at any (non singular) point of tangency of $\mathcal S_\infty$ 
with $\mathcal U_\infty$, the contact is ordinary,
\item[(iv)] at any singular point of $\mathcal S_\infty$
contained in $\mathcal U_\infty$, the tangent line to $\mathcal U_\infty$ is not contained
in the tangent cone to $\mathcal S_\infty$.
\end{itemize}
Then 
\[
\mathfrak{n}_{\mathcal{S}}=c_\nu(\mathcal S).\sigma_2+d_{\mathcal S}(d_{\mathcal S}-1).\sigma _{1,1}\in A^{2}(\mathbb{G}(1,3))
\]
and 
the normal class of $\mathcal S$ is
$$c_\nu(\mathcal S)=d_{\mathcal S}^3-d_{\mathcal S}^2+d_{\mathcal S}-\sum_{k\ge 2}((k-1)^2m_\infty^{*(k)}+
   k(k-1)\tilde m_\infty^{(k)}) -2\kappa_\infty^* - 3\tilde \kappa_\infty-c_\infty,$$
where
\begin{itemize}
\item $m_\infty^{*(k)}$ (resp. $\tilde m_\infty^{(k)}$) is the number of ordinary multiple points of
order $k$ of $\mathcal S_\infty$ outside (resp. contained in) $\mathcal U_\infty$,
\item $\kappa_\infty^*$ (resp. $\tilde \kappa_\infty$) 
is the number of ordinary cusps of $\mathcal S_\infty$ outside (resp. contained in) $\mathcal U_\infty$,
\item $c_\infty$ is the number of ordinary (non singular) points of tangency of $\mathcal S_\infty$ with $\mathcal U_\infty$.
\end{itemize}
\end{thm}
\begin{exa}[n=3]
The surface $\mathcal S=V(xzt-tx^2-zt^2-xz^2+y^3)\subset \mathbb P^3$
is smooth, its only point of tangency with $\mathcal H_\infty=V(t)$ is
$P[1:0:0:0]$ which is an ordinary cusp of $\mathcal S_\infty=V(t,-xz^2+y^3)$. Moreover $\mathcal S_\infty$ has no point of tangency with
$\mathcal U_\infty$. Hence the normal class of $\mathcal S$ is
$27-9+3-2=19$.
\end{exa}
Theorem \ref{thmhypersurface} (resp. \ref{thmsurfaces}) is a consequence of Theorem \ref{formulegeneralehypersurface} (resp. \ref{formulegeneralesurface}).
In a more general setting, when $n=3$, we can replace $\alpha_{\mathcal S}$ in $\tilde\alpha_{\mathcal S}=\frac{\alpha_{\mathcal S}}H$
(for some homogeneous polynomial $H$ of degree $d_H$) so that the
set $\tilde{\mathcal B}^{(0)}_{\mathcal S}$ of base points of $\tilde\alpha_{\mathcal S}$
has dimension at most 1. In this case, we consider a notion of
normal polars associated to $\tilde\alpha_{\mathcal S}$ which have generically  
dimension 1 and degree $\tilde d_{\mathcal S}^2
-\tilde d_{\mathcal S}+1$
(with $\tilde d_{\mathcal S}=d_{\mathcal S}-d_H$).
\begin{thm}[n=3]\label{factorisable}
Let $\mathcal S$ be an irreducible surface of $\mathbb P^3$.
If the set $\tilde{\mathcal B}^{(0)}_{\mathcal S}\cap\mathcal S$ is finite, 
then 
$\mathfrak{n}_{\mathcal{S}}=c_\nu(\mathcal S).\sigma_2+d_{\mathcal S}(\tilde d_{\mathcal S}-1)
.\sigma _{1,1}\in A^{2}(\mathbb{G}(1,3))$ and
 the normal class $c_\nu(\mathcal S)$ 
of $\mathcal S$ is equal to 
$d_{\mathcal S}(\tilde d_{\mathcal S}^2-\tilde d_{\mathcal S}+1) $ 
minus the intersection multiplicity of $\mathcal S$
with its generic normal polars  $\tilde{\mathcal P}_{A,\mathcal S}$  at points $m\in\tilde{\mathcal B}^{(0)}_{\mathcal S}\cap \mathcal S$.
\end{thm}
When the surface is a "cylinder" or a surface of revolution,  its normal class is equal to
the normal class of its plane base curve. The normal class of any plane curve
is given by the simple formula of Theorem \ref{thmcurves} below, that we give for completness.
Let us recall that, when $\mathcal C=V(F)$ is an irreducible curve of $\mathbb P^2$, 
the evolute of $\mathcal C$ is the curve tangent to the family of normal lines to $\mathcal Z$
and that the evolute of a line or a circle is reduced to a single point.
Hence, except for lines and circles, the normal class of $\mathcal C$ is simply the class (with multiplicity) 
of its evolute.
The following result generalizes the result by Salmon \cite[p. 137]{Salmon-Cayley} proved in the case
of Pl\"ucker curves (plane curves with no worse multiple tangents than ordinary double tangents, no singularities other than
ordinary nodes and cusps) to any plane curve (with any type of singularities).
We write $\ell_\infty$ for the line at infinity of $\mathbb P^2$.
We define the two cyclic points $I[1:i:0]$ and $J[1:-i:0]$ in $\mathbb P^2$ (when $n=2$, $\mathcal U_\infty=\{I,J\}$).
\begin{thm}[n=2]\label{thmcurves}
Let $\mathcal C=V(F)$ be an irreducible curve of $\mathbb P^2$ of degree $d\ge 2$ with class $d^\vee$. Then
its normal class is $$c_\nu(\mathcal C)=d+d^\vee-\Omega(\mathcal C,\ell_\infty)-\mu_{I}(\mathcal C)-\mu_{J}(\mathcal C),$$
where $\Omega$ denotes the sum of the contact numbers between two curves
and where $\mu_P(\mathcal C)$ is the multiplicity of $P$ on $\mathcal C$.
\end{thm}
In \cite{Fantechi}, Fantechi proved that the evolute map is birational from $\mathcal C$ to
its evolute curve unless if\footnote{We write $[x:y:z]$ for the coordinates of
$m\in\mathbb P^2$ and $F_x,F_y,F_z$ for the partial derivatives of $F$.} 
$F_x^2+F_y^2$ is a square modulo $F$ and that in this latest
case the evolute map is $2:1$ (if $\mathcal C$ is neither a line nor a circle).
Therefore, the normal class $c_\nu(\mathcal C)$ of a plane curve $\mathcal C$ corresponds to the class of its evolute
unless $F_x^2+F_y^2$ is a square modulo $F$ and in this last case, the normal class $c_\nu(\mathcal C)$ of $\mathcal C$
corresponds to the class of its evolute times 2 (if $\mathcal C$ is neither a line nor a circle).

The notion of focal loci generalizes the notion of evolute to higher dimension \cite{Trifogli,CataneseTrifogli}.
The normal lines of an hypersurface $\mathcal Z$ are tangent to the focal loci hypersurface of $\mathcal Z$ but of course the normal class of $\mathcal Z$
does not correspond anymore (in general) to the class of its focal loci (the normal lines to $\mathcal Z$ are contained
in but are not equal to the tangent hyperplanes of its focal loci).

In Section \ref{SEC00}, we introduce normal lines, normal class, normal polars in $\mathbb P^n$ (see also Appendix \ref{NORMAL} for the link between projective orthogonality and affine orthogonality).
In Section \ref{SECpolar}, we study normal polars and prove Theorems \ref{formulegeneralehypersurface} and \ref{thmhypersurface}. 
In Section \ref{incidenceschubert}, we introduce the orthogonal incidence variety $\mathcal I^\perp$ in $\mathbb G(1,n)$, give some recalls on the Schubert classes in the Chow ring of $\mathbb G(1,3)$
and prove Theorems \ref{formulegeneralesurface}
and \ref{factorisable}.
In Section \ref{sec:proofthm1}, we prove Theorem \ref{thmsurfaces}. 
In Section \ref{secquadric}, we apply our results on examples in $\mathbb P^3$: we compute the normal class 
of every quadric and of a cubic 
surface with singularity $E_6$.
In Section \ref{proofcurve}, we prove Theorem \ref{thmcurves}.
Appendix \ref{cylindreetrevolution} on the normal class of "cylinders" and of surfaces of revolution in $\mathbb P^n$.
\section{Normal lines, normal class and normal polars}\label{SEC00}
\subsection{Definitions and notations}\label{DEFI0}
Let $\mathbf V$ be a $\mathbb C$-vector space of dimension $n+1$.
Given $\mathcal Z=V(F)\ne\mathcal H^\infty$ an irreducible hypersurface of 
$\mathbb P^n=\mathbb P(\mathbf V)$ (with
$F\in Sym(\mathbf{V}^\vee)\cong\mathbb C[x_1,...,x_{n+1}]$), we consider
the rational map $n_{\mathcal Z}:\mathbb P^n\dashrightarrow\mathcal H^\infty$
given by $n_{\mathcal Z}=[F_{x_1}:\cdots :F_{x_n}:0]$.
Note that, for nonsingular $m\in \mathcal Z$ such that 
the tangent hyperplane $\mathcal T_m\mathcal Z$ to $\mathcal Z$ at $m$
is not $\mathcal H^\infty$, $n_{\mathcal Z}(m)$ is the pole of the $(n-2)$-variety at infinity  $\mathcal T_m\mathcal Z\cap\mathcal H^\infty\subset\mathcal H^\infty$ 
with respect to the {\bf umbilical}
$\mathcal U_\infty:=V(x_1^2+...+x_n^2)\cap\mathcal H^\infty\subset\mathcal H^\infty$. $\mathcal U_\infty$ corresponds to the set of {\bf  circular points at infinity}.
\begin{defi}
{\bf The projective normal line} $\mathcal N_m\mathcal Z$ to $\mathcal Z$
at $m\in\mathcal Z$ is the line $(m\, n_{\mathcal Z}(m))$ when $n_{\mathcal Z}(m)$ 
is well defined in $\mathbb P^n$ and not equal to $m$.
\end{defi}
\begin{rqe}
This is a generalization of affine normal lines in the euclidean space $E_n$. Indeed, 
if $F$ has real coefficients and if $m\in\mathcal Z\setminus\mathcal H_\infty$ 
has real coordinates $[x^{(0)}_1:\cdots:x^{(0)}_n:1]$, 
then $\mathcal N_m\mathcal Z$ corresponds to the affine normal line of the affine
hypersurface $V(F(x_1,...,x_n,1))\subset E_n$
at the point of coordinates $(x^{(0)}_1,\cdots,x^{(0)}_n)$ (see Section \ref{NORMAL}).
\end{rqe}
The aim of this work is the study of the notion of normal class.
\begin{defi}
Let $\mathcal Z$ be an irreducible hypersurface of $\mathbb P^n$.
{\bf The normal class} of $\mathcal Z$ is the number $c_\nu(\mathcal Z)$ of $m\in\mathcal Z$
such that $\mathcal N_m(\mathcal Z)$ contains $m_1$ for a generic
$m_1\in\mathbb P^n$.
\end{defi}
Let $\Delta:=\{(m_1,m_2)\in\mathbb P^n\times\mathbb{P}^{n}\ :\ m_1=m_2\}$
be the diagonal of $\mathbb{P}^{n}\times\mathbb{P}^{n}$.
Recall that the {\bf Pl\"ucker embedding} 
$\left(  \mathbb{P}^{n}\times\mathbb{P}^{n}\right)  \backslash\Delta 
\overset{Pl}{\hookrightarrow}  \mathbb{P}
(\bigwedge^{2}\mathbf V)\cong \mathbb{P}^{\frac{n(n+1)}2-1}$ is defined by
$$Pl(u,v)=\bigwedge^2({u}, {v})=\left[p_{i,j}=u_iv_j-u_jv_i\right]
    _{1\le i<j\le n+1}\in\mathbb P^{\frac{n(n+1)}2-1},$$
with $p_{i,j}=-p_{j,i}$ the $(i,j)$-th Pl\"ucker coordinate, identifying $\mathbb P^{\frac{n(n+1)}2-1}$ with the projective space of $n\times n$ antisymmetric matrices.
Its image is the Grassmannian $\mathbb G(1,n)$ (see \cite{Eisenbud-Harris}) given by
$$\mathbb G(1,n):=Pl((\mathbb P^n)^2\setminus \Delta)=\bigcap_{(i,j_1,j_2,j_3)\in\mathcal I}V(B_{i,j_1,j_2,j_3}) \subset \mathbb P^{\frac{n(n+1)}2-1},$$
where $B_{i,j_1,j_2,j_3}:=p_{i,j_1}p_{j_2,j_3}-p_{i,j_2}p_{j_1,j_3}+p_{i,j_3}p_{j_1,j_2}$ and where $\mathcal I$ is the set of $(i,j_1,j_2,j_3)\in\{1,...,n+1\}$ such that
$j_1<j_2<j_3$ and $j_1,j_2,j_3\ne i$. We recall also that $\dim \mathbb G(1,n)=2n-2$.
\begin{rqe}
Let $h_{\mathcal Z}:\mathbb{P}^{n}\setminus V(F_{x_1},\cdots,F_{x_n})\rightarrow \mathbb{P}^{n}\times\mathbb{P}^{n}$
be the morphism defined by $j(m)=\left(  m,{n}_{\mathcal Z}(m)\right)  $.
The variety $\mathfrak N_{\mathcal Z}\subset\mathbb G(1,n)$ of projective normal lines to $\mathcal Z$ is the (Zariski closure of the) image of $\mathcal Z$
by the regular map $\alpha_{\mathcal Z}:=Pl\circ h_{\mathcal Z}:\mathbb P^n\setminus\mathcal B^{(0)}_{\mathcal Z}\rightarrow \mathbb{P}^{\frac{n(n+1)}2-1}$, with $\mathcal B^{(0)}_{\mathcal Z}:=V(F_{x_1},...,F_{x_n})\cup j^{-1}(\Delta)$, i.e.
$$\mathcal B^{(0)}_{\mathcal Z}:=\left\{  m\in\mathbb{P}^n;\ \bigwedge^2(\mathbf {m}, \mathbf n_{\mathcal Z}(\mathbf m))=\mathbf 0\ \mbox{in}\ \bigwedge^2\mathbf V\right\}.$$
\end{rqe}
Note that the number of normal lines to $\mathcal Z$ passing through 
$A\in\mathbb P^n$
corresponds to the number of $m\in\mathcal Z\setminus\mathcal B_{\mathcal Z}$ satisfying the following
sets of equations~:
\begin{equation}\label{EQUA000}
\bigwedge^3[\mathbf {m}\ \mathbf n_{\mathcal Z}(\mathbf m)\ \mathbf A] =\mathbf 0\quad\mbox{in}\ \bigwedge^3 \mathbf V.
\end{equation} 
\begin{defi}
For any $A\in\mathbb P^n$, 
the set of points $m\in \mathbb P^n$ satisfying \eqref{EQUA000} is called
{\bf normal polar} $\mathcal P_{A,\mathcal Z}$ of $\mathcal Z$ with respect to $A$
\end{defi}
\subsection{Projective similitudes}
Recall that, for every field  $\mathbb{\Bbbk }$,
$$GO(n,\mathbb{\Bbbk }
)=\left\{ A\in GL(n,\mathbb{\Bbbk });\exists \lambda \in \mathbb{\Bbbk }
^{\ast },A\cdot^{t}A=\lambda \cdot I_{n}\right\} $$
is the \textbf{orthogonal similitude group} (for the standard products) and that 
$GOAff(n,\mathbb{\Bbbk })=\mathbb{\Bbbk }^{n}\rtimes GO(n,\mathbb{\Bbbk })$
is the \textbf{orthogonal similitude affine group}.
We have a natural monomorphism of groups
$\kappa :Aff(n,\mathbb{R})=\mathbb{R}^{n}\rtimes GL(n,\mathbb{R}
)\longrightarrow GL(n+1,\mathbb{R})$
given by
\begin{equation}\label{similitude}
\kappa (b,A)=\left( 
\begin{array}{ccccccc}
a_{11} & ... &  &  & ... & a_{1n} & b_{1} \\ 
a_{21} & ... &  &  & ... & a_{2n} & b_{2} \\ 
&  &  &  &  &  &  \\ 
a_{n1} & .. &  &  & ... & a_{nn} & b_{n} \\ 
0 & ... &  &  & 0 & 0 & 1
\end{array}
\right) 
\end{equation}
and,
by restriction, 
$\kappa |_{GOAff(n,\mathbb{R})}:GOAff(n,\mathbb{R})=\mathbb{R}^{n}\rtimes
GO(n,\mathbb{R})\longrightarrow GL(n+1,\mathbb{R})$. Analogously we  have a natural monomorphism of groups 
$\kappa ^{\prime }:=(\kappa\otimes 1) |_{GOAff(n,\mathbb{C})}:GOAff(n,\mathbb{C})=
\mathbb{C}^{n}\rtimes GO(n,\mathbb{C})\longrightarrow GL(n+1,\mathbb{C})$. 
Composing with the canonical projection $\pi :GL(n+1,\mathbb{C}
)\longrightarrow \mathbb{P}(GL(n+1,\mathbb{C}))$ we obtain the \textbf{
projective  complex  similitude Group}:
\begin{equation*}
\widehat{Sim_{\mathbb{C}}(n)}:=(\pi \circ \kappa ^{\prime })(G OAff(n,\mathbb{
C})).
\end{equation*}
which acts naturally on $\mathbb{P}^{n}$.
\begin{defi}
An element of $\mathbb P(Gl(\mathbf V))$ corresponding
to an element of $\widehat{Sim_{\mathbb{C}}(n)}$ with respect to the basis
$(\mathbf e_1,\cdots,\mathbf e_n)$
is called a \textbf{projective similitude of }$\mathbb P^{n}.$
\end{defi}
The set of projective similitudes of $\mathbb P^n$ is isomorphic to $\widehat{Sim_{\mathbb{C}}(n)}$.
\begin{lem}\label{lemmesimilitude}
The projective similitude preserves the orthogonality structure in $\mathbb P^n$. They preserve namely the normal lines
and the normal class of surfaces of $\mathbb P^n$.
\end{lem}
This lemma has a straightforward proof that is omitted.
\section{Proof of Theorem \ref{formulegeneralehypersurface}}\label{SECpolar}
\subsection{Geometric study of $\mathcal B_{\mathcal Z}:=\mathcal B^{(0)}_{\mathcal Z}\cap{\mathcal Z}$}\label{sec:base}
We write $\mathcal Z_\infty:=\mathcal Z\cap\mathcal H^\infty$. Recall that $\mathcal U_\infty:=\mathcal H^\infty\cap V\left(x_1^2+...+x_n^2\right)$.
\begin{prop}
A point of $\mathcal Z$ is in $\mathcal B_\mathcal Z$ if it is a singular point 
of $\mathcal Z$ or a tangential point of $\mathcal Z$ at infinity or a tangential point of $\mathcal Z_\infty$ to the umbilical, i.e.
$\mathcal B_{\mathcal Z}=\sing(\mathcal Z)\cup \mathcal K_\infty(\mathcal Z)\cup
  \Gamma_\infty(\mathcal Z)$, where
\begin{itemize}
\item $\sing(\mathcal Z)$ is the set of singular points of
$\mathcal Z$,
\item $\mathcal K_\infty(\mathcal Z)$ is the set of points of $\mathcal Z$ at which the tangent hyperplane is $\mathcal H^\infty$,
\item $\Gamma_\infty(\mathcal Z)$ is
the set of points of $\mathcal Z_\infty\cap \mathcal U_\infty$ 
at which the tangent space to $\mathcal Z_\infty$ and to $\mathcal U_\infty$ are the same.
\end{itemize}
\end{prop}
\begin{proof}
Let $m\in\mathcal Z$. We have
\begin{eqnarray*}
m\in\mathcal B_{\mathcal Z} &\Leftrightarrow& \bigwedge^2\left(\mathbf m , \mathbf n_{\mathcal Z}(\mathbf m)\right)=0\\
&\Leftrightarrow& \mathbf n_{\mathcal Z}(\mathbf m)=\mathbf 0\ \mbox{or}\ m= n_{\mathcal Z}( m)\\
&\Leftrightarrow& m\in V(F_{x_1},\cdots,F_{x_n})\ \mbox{or}\ m= n_{\mathcal Z}( m).
\end{eqnarray*}
Now $m\in V(F_{x_1},\cdots,F_{x_n})$ means either that $m$ is a singular point of $\mathcal S$ or that
$\mathcal T_m\mathcal Z=\mathcal H^\infty$.

Let $m=[x_1:\cdots:x_{n+1}]\in\mathcal Z$ be such that $m=n_{\mathcal S}( m)$. So $[x_1:\cdots:x_{n+1}]=[F_{x_1}:\cdots: F_{x_n}:0]$. In particular
$x_{n+1}=0$.
Due to the Euler identity, we have
$0=\sum_{i=1}^{n+1}x_iF_{x_i}=\sum_{i=1}^{n+1} x_i^2$. Hence $m\in\mathcal U_\infty$.
Note that the $(n-2)$-dimensional tangent space $\mathcal T_m\mathcal U_\infty$ to $\mathcal U_\infty$ at $m$ has equations 
$X_{n+1}=0$ and $\langle m,\cdot\rangle$
and that the $(n-2)$-dimensional tangent space $\mathcal T_m\mathcal Z_\infty$ to $\mathcal Z_\infty$ at $m$ has equations 
$X_{n+1}=0$ and $\langle(n_{\mathcal S}(m),\cdot\rangle$.
We conclude that $\mathcal T_m\mathcal U_\infty=\mathcal T_m\mathcal Z_\infty$. 

Conversely, if $m=[x_1:\cdots:x_n:0]$ is a nonsingular point of $\mathcal Z_\infty\cap\mathcal U_\infty$ such that
$\mathcal T_m\mathcal U_\infty=\mathcal T_m\mathcal Z_\infty$, then the linear spaces $Span(\mathbf{m},\vec e_{n+1})$
and $Span(\nabla F,\vec e_{n+1})$ are equal which implies that $[x_1:\cdots:x_n:0]=[F_{x_1}:\cdots:F_{x_n}:0]$.
\end{proof}
Recall that the dual
variety of $\mathcal Z_\infty\subset\mathcal H^\infty$ is the variety
$\mathcal Z_\infty^\vee\subset (\mathcal H^\infty)^\vee\cong (\mathbb P^{n-1})^\vee$ of tangent hyperplanes
to $\mathcal Z_\infty$.
It corresponds to the (Zariski closure of the) image of $\mathcal Z_\infty$ by the rational map $n_{\mathcal Z}$. We write $\mathcal Z_\infty^\wedge\subset\mathbb P^n$ for this image. With this notation,
$\mathcal B_{\mathcal Z}=\sing(\mathcal Z)\cup (\mathcal Z_\infty\cap\mathcal Z_\infty^\wedge)$.

\begin{rqe}\label{Basegenerique}
For a generic hypersurface of $\mathbb P^n$, $\mathcal B_{\mathcal Z}=\emptyset$
and so $\dim \mathcal B_{\mathcal Z}^{(0)}\le 0$.
\end{rqe}
But we will also consider cases for which $\#\mathcal B_{\mathcal Z}<\infty$, and so
$\dim\mathcal B^{(0)}_{\mathcal Z}\le 1$.
\begin{exa}[n=3]\label{exemple2}
For the saddle surface $\mathcal S_1=V(xy-zt)$, the set 
$\mathcal B_{\mathcal S_1}$ contains a single point $[0:0:1:0]$
which is a point of tangency at infinity of $\mathcal S_1$.

For the ellipsoid $\mathcal E_1:=V(x^2+2y^2+4z^2-t^2)$, the set $\mathcal B_{\mathcal E_1}$ is empty.

For the ellipsoid $\mathcal E_2:=V(x^2+4y^2+4z^2-t^2)$, the set 
$\mathcal B_{\mathcal E_2}$ has two elements: $[0:1:\pm i:0]$ which 
are points of tangency of $\mathcal E_2$ with $\mathcal U_\infty$.
\end{exa}
\begin{exa}
For the cuartic $\mathcal Z:=V(x_1^2+x_2^2+(x_3+x_5)x_3+(2x_3+x_4)x_4)\subset \mathbb P^4$, $Sing(\mathcal Z)=\emptyset$, $\mathcal K_\infty(\mathcal Z)=\{[0:0:1:-1:0]\}$ and $\Gamma_\infty(\mathcal Z)=\{I_1,I_2\}$,
with $I_1[1:i:0:0:0]$ and  $I_2[1:-i:0:0:0]$.
\end{exa}

\subsection{Normal polars of $\mathcal Z\subset\mathbb P^n$}
\label{polars}
Let $\mathcal Z=V(F)\subset \mathbb P^n$ (with $F\in Sym(\mathbf V^\vee)$) be an
irreducible hypersurface.
For every $A\in\mathbb P^n$, {\bf the normal polar} $\mathcal P_{A,\mathcal Z}$ of $\mathcal Z$ with respect to $A$ is the set of $m\in\mathbb P^n$ satisfying
the $\left(\begin{array}{c}n+1\\ 3\end{array}\right)$ equations of \eqref{EQUA000}.
For every $m,A\in \mathbb P^n$, we have
$$m\in  \mathcal P_{A,\mathcal Z}\\ \Leftrightarrow\ m\in\mathcal B^{(0)}_{\mathcal Z}\ \mbox{or}\ A\in 
\mathcal N_m\mathcal Z,$$
extending the definition of $\mathcal N_m\mathcal Z$ from $m\in\mathcal Z$ 
to $m\in\mathbb P^n$.
\begin{lem}[The projective similitudes preserve the normal polars]
\label{preservpolar}
Let $\mathcal Z=V(F)\subset\mathbb P^n$ be a hypersurface and
$\varphi$ be any projective similitude, then 
$\varphi(\mathcal P_{A,\mathcal Z})=\mathcal P_{\varphi(A),\varphi(\mathcal Z)}$.
\end{lem}
\begin{proof}
Due to Lemma \ref{lemmesimilitude}, $\varphi(\mathcal N_m\mathcal Z)=\mathcal N_{\varphi(m)}
(\varphi(\mathcal Z))$ which gives the result.
\end{proof}
Note that 
$$\mathcal P_{A,\mathcal Z}=\mathcal B^{(0)}_{\mathcal Z}\cup\left(\bigcap_{i<j<k}\alpha_{\mathcal Z}^{-1}\mathcal H_{A,i,j,k}\right),$$
where $\mathcal{H}_{A,i,j,k}$ is the hyperplane of $\mathbb{P}^{\frac{n(n+1)}2-1}$ given by
$\mathcal{H}_{A,i,j,k}:= V(D_{i,j,k})\subset \mathbb{P}^{\frac{n(n+1)}2-1}$, with
$D_{i,j,k}:=a_ip_{j,k}-a_jp_ {i,k}+a_kp_{i,j}$. 
On  $\mathbb G(1,n)$,  $p=Pl(u,v)\in\bigcap_{i<j<k}\mathcal H_{A,i,j,k}$
means that  $\bigwedge ^3(\mathbf A,\mathbf u,\mathbf v)=0$.
\begin{lem}\label{lemdimG(1,n)}
For every $A\in \mathbb P^n$, the set $\bigcap_{i<j<k}\mathcal H_{A,i,j,k}$ is a
$(n-1)$-dimensional linear space of $\mathbb P^{\frac{n(n+1)}2-1}$ contained
in $\mathbb G(1,n)$.
\end{lem}
\begin{proof}
Let $A[a_1:\cdots:a_{n+1}]\in\mathbb P^n$.
Assume for example $a_{j_0}\ne 0$ (the proof being analogous when $a_j\ne 0$ 
for symetry reason).

Let $p\in\bigcap_{i<j<k}\mathcal H_{A,i,j,k}$. Let us prove that $p\in\mathbb G(1,n)$.
Let $i,j_1,j_2,j_3\in\{1,...,n+1\}$ be distinct indices.
Due to $D_{j_1,j_2,j_3}=D_{i,j_1,j_2}=0$, we have
\begin{multline*}
a_{j_1}a_{j_2}p_{i,j_1}p_{j_2,j_3}=a_{j_1}a_{j_2}p_{j_1,j_3}p_{i,j_2}
    +a_{j_1}p_{j_1,j_2}(-a_{j_3}p_{i,j_2})\\+a_{j_2}p_{j_1,j_2}(-a_ip_{j_1,j_3})+
      p_{j_1,j_2}(a_ia_{j_3}p_{j_1,j_2})\\
=a_{j_1}a_{j_2}p_{j_1,j_3}p_{i,j_2}-a_{j_1}a_{j_2}p_{i,j_3}p_{j_1,j_2}
      -a_jp_{j_1,j_2}D_{i,j_2,j_3}+a_ip_{j_1,j_2}D_{j_1,j_2,j_3}.
\end{multline*}
Hence $a_{j_1}a_{j_2}B_{i,j_1,j_2,j_3}=0$, for every $i,j_1,j_2,j_3\in\{1,...,n+1\}$.
So $B_{i,j_1,j_2,j_3}\ne 0$ implies that $a_{j_1}=a_{j_2}=a_{j_3}=0$ (up to a permutation 
of $(i,j_1,j_2,j_3)$)  and so  
$0=D_{j_0,j_1,j_2}=D_{j_0,j_2,j_3}=D_{j_0,j_1,j_3}$ imply $p_{j_1,j_2}=p_{j_2,j_3}=p_{j_1,j_3}=0$ which contradicts $B_{i,j_1,j_2,j_3}\ne 0$.
Since $\bigwedge ^4(\mathbf A,\mathbf A,\mathbf u,\mathbf v)=0$, we get 
that $a_{j_0}D_{i,j,k}-a_iD_{j_0,j,k}+a_jD_{j_0,i,k}-a_kD_{j_0,i,j}=0$ for every $1\le i<j<k\le n+1$ such that $i,j,k\ne j_0$. 
Hence
$\bigcap_{i<j<k}\mathcal H_{A,i,j,k}=\bigcap_{i<j, i,j\ne j_0}\mathcal H_{A,j_0,i,j}$.
Since $a_{1}\ne 0$, the $\frac {n(n-1)}2$ corresponding linear equations are linearly independent and so
$\bigcap_{i<j, i,j\ne j_0}\mathcal H_{A,j_0,i,j}\subset \mathbb P^{\frac{n(n+1)}2-1}$
has dimension 
$\frac{n(n+1)}2-1-\frac{(n-1)n}2=n-1$.
\end{proof}
\begin{prop}\label{degrepolaire}
Let $\mathcal Z=V(F)\subset\mathbb P^n$ be an irreducible hypersurface
such that $d_{\mathcal Z}:=\deg\mathcal Z\ge 2$ and $\dim\mathcal B^{(0)}_{\mathcal Z}\le 1$.
Then, for a generic $A\in\mathbf V$, we have $\dim \mathcal P_{A,\mathcal Z}=1$ and 
$$\deg \mathcal P_{A,\mathcal Z}=\sum_{k=0}^{n-1}(d_{\mathcal Z}-1)^k.$$
\end{prop}
\begin{proof}
Due to the proof of Lemma \ref{lemdimG(1,n)}, for every $A\in\mathbb P^n$
such that $a_{n+1}\ne 0$, we have
$\bigcap_{i<j<k}\mathcal H_{A,i,j,k}=\bigcap_{i<j<n+1}\mathcal H_{A,n+1,i,j}\subset\mathbb P^{\frac{n(n+1)}2-1}$
and so 
$$\mathcal P_{A,\mathcal Z}= \bigcap_{1\le i<j\le n}V(E_{A,i,j})\subset\mathbb P^n,$$
with
$$\forall i,j\in\{1,...,n\},\quad E_{A,i,j}:=L_{A,i}F_{x_j}-L_{A,j}F_{x_i}\quad\mbox{and}\quad L_{A,i}:=a_{n+1}x_i-a_ix_{n+1},$$
i.e. $E_{A,i,j}=a_{n+1}(x_iF_{x_j}-x_jF_{x_i})+a_jx_{n+1}F_{x_i}-a_ix_{n+1}F_{x_j}$.
Note that 
\begin{equation}\label{simplification}
L_{A,k}E_{A,i,j}-L_{A,j}E_{A,i,k}=L_{A,i}E_{A,k,j}\quad\mbox{and}\quad 
F_{x_k}E_{A,i,j}-F_{x_j}E_{A,i,k}=F_{x_i}E_{A,k,j}.
\end{equation}
Hence 
\begin{equation}\label{n-1eq}
\forall i\in\{1,...,n\},\quad
\mathcal P_{A,\mathcal Z}\setminus V(L_{A,i},F_{x_i})=\bigcap_{j\in\{1,...,n\}\setminus\{i\}}
V(E_{A,i,j})\setminus  V(L_{A,i},F_{x_i}),
\end{equation}
and so $\dim\mathcal P_{A,\mathcal Z}\ge 1$.
Recall that $\mathcal B_{\mathcal Z}^{(0)}=\bigcap_{i=1}^nV(x_{n+1}F_{x_i})\cap
\bigcap_{i,j=1}^nV(x_iF_{x_j}-x_jF_{x_i})$ and that we have assumed that
$\dim \mathcal B_{\mathcal Z}^{(0)}\le 1$. In particular
$\mathcal P_{A,\mathcal Z}\cap\mathcal H^\infty=\mathcal B_{\mathcal Z}^{(0)}\cap
\mathcal H^\infty$ and 
$\mathcal B_{\mathcal Z}^{(0)}\setminus\mathcal H^\infty
=V(F_{x_1},...,F_{x_n})\setminus\mathcal H^\infty$.
This combined with \eqref{n-1eq} and with the expression of $E_{A,i,j}$ leads to $\dim\mathcal P_{A,\mathcal Z}= 1$.

Now let us compute the degree of $\mathcal P_{A,\mathcal Z}$.
The idea is to prove an induction formula. 
Assume that $A\not\in\mathcal H^\infty$ is such that $\dim\mathcal P_{A,\mathcal Z}
=1$. Let $\mathcal H=V(\sum_{i=1}^{n+1}
\alpha_ix_i)\subset\mathbb P^n$ be an hyperplane 
such that $\#(\mathcal H\cap \mathcal P_{A,\mathcal Z})
=\deg \mathcal P_{A,\mathcal Z}$ and $\sum_{i=1}^n\alpha_i^2\ne 0$.
We compose by a projective similitude $\phi:\mathbb P^n\rightarrow\mathbb P^n$ 
so that $\phi(A)$ has projective coordinates $[0:\cdots:0:1]$ and that $\hat{\mathcal H}:=\phi(\mathcal H)=V(x_1-b x_{n+1})\subset\mathbb P^n$. Set 
$\hat{\mathcal Z}:=\phi(\mathcal Z)=V(\hat F)\subset\mathbb P^n$, with $\hat F:= F\circ\phi$.
Hence $\phi(\mathcal P_{A,\mathcal Z})=\mathcal P_{\phi(A),\tilde{\mathcal Z}}$
is the set of points $m[x_1:\cdots:x_{n+1}]\in\mathbb P^n$
such that $\bigwedge^2 \left(\left(\begin{array}{c}x_1\\ \vdots\\x_n\\0\end{array}\right),\left(
\begin{array}{c}\hat F_{x_1}\\ \vdots\\\hat F_{x_n}\\0\end{array}\right)\right)=0$ in $\bigwedge^2\mathbf V$.
We then define $G(x_2,...,x_{n+1}):=\hat F(bx_{n+1},x_2,...,x_{n+1})
\in\mathbb C[x_2,...,x_{n+1}]$ and $H(x_3,...,x_{n+1}):=G(0,x_3,...,x_{n+1})\in\mathbb C[x_3,...,x_{n+1}]$.
We set $\mathcal Z_1:=V(G)\subset\mathbb P^{n-1}$,  
$\mathcal Z_2:=V(H)\subset\mathbb P^{n-2}$ and $B_k[0:...:0:1]\in\mathbb P^{k}$.
We then write $\mathcal P_{n-k,B_{n-k},{\mathcal Z}_k}$ for the normal polar
in $\mathbb P^{n-k}$ of $\mathcal Z_k\subset\mathbb P^{n-k}$ with respect to 
$B_{n-k}$, with the conventions $\mathcal P_{0,B_0,{\mathcal Z}_k}=\emptyset$
(if $k=n$) and $\mathcal P_{1,B_1,{\mathcal Z}_k}=\mathbb P^1$ (if $k=n-1$).
We will prove that
\[
\deg \mathcal P_{A,{\mathcal Z}}=d\times \deg \mathcal P_{n-1,B_{n-1},{\mathcal Z}_1} - (d-1)\times \deg \mathcal P_{n-2,B_{n-2},{\mathcal Z}_2}\, .
\]
Let $\Pi_1: \mathbb P^n\rightarrow \mathbb P^{n-1}$
and $\Pi_2:\mathbb P^n\rightarrow \mathbb P^{n-2}$ be
the projections given by $\Pi_1[x_1:...:x_{n+1}]=[x_2:...:x_{n+1}]$
and $\Pi_2[x_1:...:x_{n+1}]=[x_3:...:x_{n+1}]$.
Due to  \eqref{simplification},
\begin{equation}\label{decomp}
\hat{\mathcal H}\cap V(x_1\hat F_{x_2}-x_2\hat F_{x_1})\cap\Pi_1^{-1}(\mathcal P_{n-1,B_{n-1},{\mathcal Z}_1})=(\hat{\mathcal H}\cap 
\mathcal P_{\phi(A),\hat{\mathcal Z}})\cup [\hat{\mathcal H}\cap  V(x_2,\hat F_{x_2})\cap\Pi_2^{-1}(\mathcal P_{n-2,B_{n-2},{\mathcal Z}_2})].
\end{equation}
For a generic $\mathcal H$ and for a good choice of $\phi$,
the union in the right hand side of \eqref{decomp} is disjoint and
\begin{eqnarray*}
\deg \mathcal P_{A,\mathcal Z}&=&\#(\mathcal H\cap \mathcal P_{A,\mathcal Z})\\
&=&\#(\hat{\mathcal H}\cap \deg \mathcal P_{\phi(A),\hat{\mathcal Z}})\\
&=& d_{\mathcal Z}.\deg \mathcal P_{n-1,B_{n-1},{\mathcal Z}_1} -(d_{\mathcal Z}-1).\mathcal P_{n-2,B_{n-2},{\mathcal Z}_2}.
\end{eqnarray*}
Hence $\deg \mathcal P_{A,\mathcal Z}=d_{\mathcal Z}$ if $n=2$ and 
 $\deg \mathcal P_{A,\mathcal Z}=d_{\mathcal Z}^2-d_{\mathcal Z}+1$ if $n=3$.
The formula in the general case follows by induction.
\end{proof}
Analogously we have the following.
\begin{prop}[n=3]\label{rqebasepoints}
If $\mathcal S$ is an irreducible algebraic surface of $\mathbb P^3$
(with projective coordinates $[x:y:z:t]$)
and if $\dim \mathcal B^{(0)}_{\mathcal S}=2$, then the two dimensional part of $\mathcal B^{(0)}_{\mathcal S}$
is $V(H)\subset\mathbb P^3$ for some homogeneous polynomial 
$H\in\mathbb C[x,y,z,t]$ of degree $d_H$.
We write $\boldsymbol{\alpha}_{\mathcal S}=H\cdot\boldsymbol{\tilde\alpha}_{\mathcal S}$.
Note that the regular map $\tilde\alpha_{\mathcal S}:\mathbb P^3\setminus\tilde{\mathcal B}^{(0)}_{\mathcal S}
\rightarrow\mathbb P^3$ (with $\dim \tilde{\mathcal B}^{(0)}_{\mathcal S}\le 1$) associated to $\boldsymbol{\tilde\alpha}_{\mathcal S}$. We then adapt our study with respect to $\tilde\alpha_{\mathcal S}$ instead 
of $\alpha_{\mathcal S}$ and define the corresponding polar $\tilde {\mathcal P}_{A,\mathcal S}$.
Then, we have $\deg \tilde {\mathcal P}_{A,\mathcal S}=( d_{\mathcal S}
        -d_H)^2-d_{\mathcal S}+d_H+1$.
\end{prop}
\begin{exa}\label{exemple1}
Note that the only irreducible quadrics $\mathcal S=V(F)\subset\mathbb P^3$
such that $\dim \mathcal B^{(0)}_{\mathcal S}\ge 2$ are the spheres and cones,
i.e. with $F$ of the following form
$$F(x,y,z,t)=(x-x_0t)^2+(y-y_0t)^2+(z-z_0t)^2+a_0t^2,$$
where $x_0,y_0,z_0,a_0$ are complex numbers (it is a sphere if $a_0\ne 0$ and it is a cone otherwise).

Hence, due to Proposition \ref{degrepolaire}, the degree of a generic normal polar of any
irreducible quadric of $\mathbb P^3$ which is neither a sphere nor a cone is 3.

Moreover, for a sphere or for a cone, applying Proposition \ref{rqebasepoints} with $H=t$, 
$\tilde{\mathcal P}_{A,S}$ is a line for a generic $A\in\mathbb P^3$.
\end{exa}
\subsection{Proof of Theorems \ref{thmhypersurface} and \ref{formulegeneralehypersurface}}
\begin{proof}[Proof of Theorem \ref{formulegeneralehypersurface}]
Let $\mathcal Z$ be an irreducible surface of $\mathbb P^n$ of degree 
$d_{\mathcal Z}\ge 2$ such that $\#\mathcal B_{\mathcal Z}<\infty$.
It remains to prove that
\begin{equation}\label{formuleclassenormalesurface}
c_\nu(\mathcal S)= d_{\mathcal Z}.\deg \mathcal P_{A,\mathcal Z}
      -\sum_{P\in\mathcal B_{\mathcal Z}}i_P\left(\mathcal Z,\mathcal P_{A,\mathcal Z}\right),
\end{equation}
for a generic $A\in\mathbb P^n$.
Note that, for a generic $A\in\mathbb P^n$, 
since $\overline{\alpha_{\mathcal Z}(\mathcal Z)}$ is irreducible of
dimension at most $n-1$, we have
$\#\bigcap_{i<j<k}\mathcal H_{A,i,j,k}\cap \overline{\alpha_{\mathcal Z}(\mathcal Z)}<\infty$ and so
$\# \mathcal P_{A,\mathcal Z}\cap \mathcal Z<\infty$ (since $\#\mathcal B_{\mathcal Z}<\infty$).
Since $\dim \mathcal P_{A,\mathcal Z}=1$ and $\# \mathcal Z\cap \mathcal P_{A,\mathcal Z}<\infty$
for a generic $A\in\mathbb P^n$, due to Proposition
\ref{degrepolaire} and to the Bezout formula, we have:
\begin{eqnarray*}
d_{\mathcal Z}.\deg \left(\mathcal P_{A,\mathcal Z}\right)&=&
\deg\left(\mathcal Z\cap \mathcal P_{A,\mathcal Z}\right)\\
&=&\sum_{P\in\mathcal B_{\mathcal Z}}i_P\left(\mathcal Z, \mathcal P_{A,\mathcal Z}\right)+
  \sum_{P\in\mathcal S\setminus\mathcal B_{\mathcal Z}}i_P\left(\mathcal Z, \mathcal P_{A,\mathcal Z}\right).
\end{eqnarray*}
Now let us prove that, for a generic $A\in\mathbb P^n$,
\begin{equation}\label{multpolaire}
\sum_{P\in\mathcal Z\setminus\mathcal B_{\mathcal Z}}i_P\left(\mathcal Z, \mathcal P_{A,\mathcal Z}\right)=
\#((\mathcal Z\cap \mathcal P_{A,\mathcal Z})\setminus\mathcal B_{\mathcal Z}).
\end{equation}
Since $\alpha_{\mathcal Z}$ defines a rational map, $\overline{\alpha_{\mathcal Z}(\mathcal Z)}$ is irreducible and its dimension is at most $n-1$.

Assume first that $\dim \overline{\alpha_{\mathcal Z}(\mathcal Z)}<n-1$.
For a generic $A\in\mathbb P^n$,
the plane $\bigcap_{i<j<k}\mathcal H_{A,i,j,k}$ does not meet 
$\overline{\alpha_{\mathcal Z}(\mathcal Z)}$ and so the left and right hand sides of \refeq{multpolaire} are both zero. So Formula \eqref{formuleclassenormalesurface}
holds true with $c_\nu(\mathcal Z)=0$.

Assume now that $\dim \overline{\alpha_{\mathcal Z}(\mathcal Z)}=n-1$. Then, for a generic $A\in\mathbb P^n$,
the plane $\bigcap_{i<j<k}\mathcal H_{A,i,j,k}$ meets $\alpha_{\mathcal Z}(\mathcal Z)$ transversally (with
intersection number 1 at every intersection point) and does not meet $\overline{\alpha_{\mathcal Z}(\mathcal Z)}\setminus\alpha_{\mathcal Z}(\mathcal Z)$.
This implies that, for a generic $A\in\mathbb P^n$, we have
$i_P\left(\mathcal Z, \mathcal P_{A,\mathcal Z}\right)=1$ for every $P\in(\mathcal S\cap\mathcal P_{A,\mathcal Z})
\setminus\mathcal B_{\mathcal Z}$ and so \refeq{multpolaire} follows.
Hence, for a generic $A\in\mathbb P^n$, we have
\begin{eqnarray*}
d_{\mathcal Z}.\deg \left(\mathcal P_{A,\mathcal Z}\right)
&=&\sum_{P\in\mathcal B_{\mathcal Z}}i_P\left(\mathcal Z, \mathcal P_{A,\mathcal Z}\right)+
    \#\{P\in \mathcal Z\setminus \mathcal B_{\mathcal Z} :\ A\in\mathcal N_m\mathcal Z\}\\
&=&\sum_{P\in\mathcal B_{\mathcal Z}}i_P\left(\mathcal S,\mathcal P_{A,\mathcal Z}\right)+c_\nu(\mathcal Z),
\end{eqnarray*}
which gives \eqref{formuleclassenormalesurface}.
\end{proof}
\begin{proof}[Proof of Theorem  \ref{thmhypersurface}]
Let $\mathcal H=V(\sum_{i=1}^{n+1}a_ix_i)$ be an hyperplane such that
$\mathcal H\ne\mathcal H^\infty$. For every $m\in\mathcal H$,
$n_{\mathcal H}(m)[a_1:\cdots:a_n:0]\in\mathbb P^n$.
Hence every $A\in\mathbb P^n$ belong to a single normal line to $\mathcal H$
(the line containing $A$ and $[a_1:\cdots:a_n:0]$).

The case $d_{\mathcal Z}\ge 2$ follows from Theorem \ref{formulegeneralehypersurface} and Remark \ref{Basegenerique}.
\end{proof}
\section{Orthogonal incidence variety and Schubert classes}\label{incidenceschubert}
\subsection{Orthogonal incidence variety}
Let us write as usual $\mathbb G(1,n)$ (resp. $\mathbb G(n-1,n)$) for the grassmannian of the lines
(resp. of the hyperplanes) of $\mathbb P^n$.
Let us write $pr_{1}:\mathbb{G}(1,n)\times \mathbb{G}(n-1,n)\rightarrow
\mathbb{G}(1,n)$ and $pr_{2}:\mathbb{G}(1,n)\times \mathbb{G}
(n-1,n)\rightarrow \mathbb{G}(n-1,n)$ for the canonical projections.
We define the {\bf orthogonal incidence variety} $\mathcal{I}^{\perp }$
by
$$\mathcal{I}^{\perp }:=\{(\mathcal{L}_1,\mathcal{H}_1)\in \mathbb G(1,n)\times 
\mathbb{G}(n-1,n)\, :\, \mathcal{L}_1{\perp }\mathcal {H}_1\}.$$
Let us write $p_{1}:\mathcal{I}^{\mathbb{\perp }}\rightarrow \mathbb{G}
(1,n)$ and $p_{2}:\mathcal{I}^{\mathbb{\perp }}\rightarrow \mathbb{G}(n-1,n)$
for the restrictions of $pr_{1}$ and  $pr_{2}$.
We want to describe in the Chow ring of $\mathbb{G}(1,n)$  and $\mathbb{G}
(n-1,n)\equiv \mathbb{P}^{n\vee }$
the rational equivalence class of $p_{2}p_{1}^{-1}(\mathcal{L)}$ and $
p_{1}p_{2}^{-1}(\mathcal{H)}$.
\begin{lem}
$p_{2}\circ p_{1}^{-1}:\mathbb{G}(1,n)\setminus\{\mathcal L\subset\mathcal H^\infty\}
\rightarrow \mathbb{G}(n-1,n)$
is a line projective bundle and 
$p_{1}\circ p_{2}^{-1}:\mathbb{G}(n-1,n)\setminus\{\mathcal H^\infty\}\rightarrow\mathbb{G}(1,n)$ and
is a plane projective bundle. 
\end{lem}
\begin{proof}
Let $\mathcal H=V(a_1x_1+\cdots a_{n+1}x_{n+1})$ be a projective hyperplane of $\mathbb P^n$, which is not $\mathcal H^\infty$.
Then 
$$ p_{1}(p_{2}^{-1}(\mathcal{H}))=\{\mathcal L\in \mathbb G(1,n),\ (a_1,\cdots,a_n,0)\in \mathbf{L}\}.$$
Moreover $ p_{1}(p_{2}^{-1}(\mathcal{H}^\infty))=\mathbb G(1,n)$.\\
Let $\mathcal L\not\subset\mathcal H^\infty$ be a line of $\mathbb P^3$, let $A_0[a_1:\cdots:a_n:0]$ be the only point in $\mathcal L\cap\mathcal H^\infty$, we have
$$ p_2(p_1^{-1}(\mathcal L))=\{\mathcal H\in\mathbb G(n-1,n)\, :\ \exists [a:b]\in\mathbb P^1,\ 
      \mathcal H=V(aa_1x_1+\cdots+aa_nx_n+bx_{n+1}) \}.$$
Finally, if $\mathcal L=\mathbb P(\mathbf L)\subset\mathcal H^\infty$ is a projective line, then we have
$$ p_2(p_1^{-1}(\mathcal L))=\{\mathcal H\in\mathbb G(n-1,n)\, :\ \exists a,b\in\mathbb C,\ \exists (a_1,\cdots,a_n,0)\in \mathbf{L},\ 
      \mathcal H=V(aa_1x_1+\cdots+aa_nx_n+bx_{n+1})   \}.$$
\end{proof}
It follows directly from the proof of this lemma that if $\mathcal H\in\mathbb G(n-1,n)\setminus\{\mathcal H^\infty\}$, the class of $p_{1}(p_{2}^{-1}(\mathcal{\mathcal H)})$ in the Chow ring $A^*(\mathbb{G}(1,n))$ is simply the Schubert class $\sigma_{n-1}$.
\subsection{Schubert classes for $\mathbb G(1,3)$}
Given a flag $\mathbf{F}=\{\mathbf{V_1}\subset \mathbf {V_2}\subset \mathbf {V_3}\subset \mathbf {V_4}=\mathbf{V}\}$ of $\mathbf V$
with dim$_{\mathbb{C}}(V_{i})=i$ for all integer $i$, we consider
its associated projective flag $\mathcal F$ of $\mathbb{P}^{3}$ (image by the canonical projection $\pi :\mathbf{V\backslash \{0\}}\rightarrow \mathbb{P}^{3}$)
\[
\mathcal{F}=\{p\in \mathcal{D}\subset \mathcal{P}\subset 
\mathbb{P}^{3}\}.
\]
Let $\mathcal{Z}^{k}$ denote the set of cycles of
codimension $k$ in $\mathbb{G}(1,3)$. We recall that the \textbf{Schubert cycles} of 
$\mathbb{G}(1,3)$ associated to $\mathcal F$ (or to {\bf F}) are given by
\begin{equation}\label{EQSchubert}
\left\{ 
\begin{array}{c}
\Sigma _{0,0}:=\mathbb{G}(1,3)\in \mathcal{Z}^{0}(\mathbb{G}(1,3)) \\ 
\Sigma _{1,0}:=\left\{ \mathcal{L}\in \mathbb{G}(1,3);\mathcal{D\cap L\neq
\varnothing }\right\} \in \mathcal{Z}^{1}(\mathbb{G}(1,3)) \\ 
\Sigma _{2,0}:=\left\{ \mathcal{L}\in \mathbb{G}(1,3);p\mathcal{\in L}%
\right\} \in \mathcal{Z}^{2}(\mathbb{G}(1,3)) \\ 
\Sigma _{1,1}:=\left\{ \mathcal{L}\in \mathbb{G}(1,3);\mathcal{L\subset P}%
\right\} \in \mathcal{Z}^{2}(\mathbb{G}(1,3)) \\ 
\Sigma _{2,1}:=\Sigma _{2,0}\cap \Sigma _{1,1}\in \mathcal{Z}^{3}(\mathbb{G}%
(1,3)) \\ 
\Sigma _{2,2}:=\left\{ \mathcal{L}\in \mathbb{G}(1,3);\mathcal{L=D}\right\}
\in \mathcal{Z}^{2}(\mathbb{G}(1,3))%
\end{array}
\right. .
\end{equation}
We write as usual $A^{\ast }(\mathbb{G}(1,3))$
for the Chow ring of $\mathbb{G}(1,3)$ and 
$\sigma _{i,j}:=\left[ \Sigma _{i,j}\right] \in A^{i+j}(\mathbb{G}(1,3))$ for \textbf{Schubert classes}.
For commodity we will use the notation $\Sigma _k:=\Sigma _{k,0}$ and $\sigma _k:=\sigma _{k,0}$.
We recall that $A^{\ast }(\mathbb{G}(1,3))$ is freely generated as graded $\mathbb{Z}$%
-module by $\left\{ \sigma _{i,j};0\leq j\leq i\leq 2\right\} $ with the following
multiplicative relations
\[
(E)\left\{ 
\begin{array}{c}
\sigma _{1,1}=\sigma _{1}^{2}-\sigma _{2} \\ 
\sigma _{1,1}\cdot \sigma _{1}=\sigma _{1}\cdot \sigma _{2}=\sigma _{2,1} \\ 
\sigma _{2,1}\cdot \sigma _{1}=\sigma _{1,1}^{2}=\sigma _{2}^{2}=\sigma
_{2,2} \\ 
\sigma _{1,1}\cdot \sigma _{2}^{{}}=0%
\end{array}
\right.  .
\]
Hence, the Chow ring of the grassmannian is
\[
A^{\ast }(\mathbb{G}(1,3))=\frac{\mathbb{Z[}\sigma _{1},\sigma _{2}]}{%
(2\sigma _{1}\cdot \sigma _{2}-\sigma _{1}^{3},\sigma _{1}^{2}\cdot \sigma
_{2}-\sigma _{2}^{2})}.
\]
\subsection{Proofs of Theorems \ref{formulegeneralesurface}
and \ref{factorisable}}
Recall that we have defined
$
\mathfrak{N}_{\mathcal{S}}:=\overline{\{\mathcal{N}_{m}(\mathcal{S});m\in \mathcal{S}\}}\subset
 \mathbb{G}(1,3)$
and $\mathfrak{n}_{\mathcal{S}}:=[\mathfrak{N}_{\mathcal{S}
}]\in A^{2}(\mathbb{G}(1,3))$.
\begin{prop}\label{PROP1}
Let $\mathcal S\subset \mathbb P^3$ be an irreducible surface of degree $d\ge 2$ of $\mathbb P^3$.
\begin{itemize}
\item If $\#\mathcal B_{\mathcal S}<\infty$, we have
\[
\mathfrak{n}_{\mathcal{S}}=c_\nu(\mathcal S).\sigma_2+d_{\mathcal S}(d_{\mathcal S}-1).\sigma _{1,1}\in A^{2}(\mathbb{G}(1,3)).
\]
\item If $\dim\mathcal B^{(0)}_{\mathcal S}=2$ with two dimensional part $V(H)$ 
and $\#\tilde{\mathcal B}_{\mathcal S}<\infty$ 
(with the notations of Proposition \ref{rqebasepoints}), then we have
\[
\mathfrak{n}_{\mathcal{S}}=c_\nu(\mathcal S).\sigma_2+d_{\mathcal S}(d_{\mathcal S}-d_H-1).\sigma _{1,1}\in A^{2}(\mathbb{G}(1,3)).
\]
\end{itemize}
\end{prop}
\begin{proof}
Since $\mathfrak{n}_{\mathcal{S}}\in A^{2}(\mathbb{G}(1,3)),$ we have $\mathfrak{n}_{\mathcal{S}
}=a.\sigma _{2}+b.\sigma _{1,1}$ for some integers $a$ and $b$.
Morever by Kleiman's transversality theorem (see for example \cite[Thm 5.20]{Eisenbud-Harris}), since $\Sigma _{1,1}:=\left\{ \mathcal{L}\in \mathbb{G}(1,3);\mathcal{L\subset P
}\right\} \in \mathcal{Z}^{2}(\mathbb{G}(1,3))$, we have $\mathfrak{n}_{\mathcal{S}
}\cdot \sigma _{1,1}=\left( a\sigma _{2}+b\sigma _{1,1}\right) \cdot \sigma_{1,1}$ and so, using
\eqref{EQSchubert}, we obtain
\begin{equation}
\mathfrak{n}_{\mathcal{S}}\cdot \sigma _{1,1}=b.\sigma _{1,1}^{2}=b.\sigma_{2,2}=b.
\end{equation}
Analogously, since
$\Sigma _{2}:=\left\{ \mathcal{L}\in \mathbb{G}(1,3);p\in \mathcal{L}
\right\} \in \mathcal{Z}^{2}(\mathbb{G}(1,3))$, due to \eqref{EQSchubert}, we have
\begin{equation}
\mathfrak{n}_{\mathcal{S}}\cdot \sigma
_{2}=\left( a\sigma _{2}^{{}}+b\sigma _{1,1}\right) \cdot \sigma _{2}=a\sigma _{2}^{2}=a\sigma_{2,2}=a .\end{equation}
Now it remains to compute $\mathfrak{n}_{\mathcal{S}}\cdot \sigma
_{2}$ and $\mathfrak{n}_{\mathcal{S}}\cdot \sigma
_{1,1}$, i.e. to compute the cardinality of the intersection of
$\mathfrak{N
}_{\mathcal{S}}$ with $\Sigma _{1,1}$ and with $\Sigma _{2}$.\\
Let us start with the computation of $a=\mathfrak{n}_{\mathcal {S}}\cdot \sigma_{2}$. 
If $\#\mathcal B_{\mathcal S}<\infty$, then, for a generic $P\in\mathbb P^3$, we have 
$$\mathfrak{N}_{\mathcal{S}}\cap \Sigma _{2}=\left\{ \mathcal{L}\in \mathfrak{N}
_{\mathcal{S}};P\in \mathcal{L}\right\} =\left\{\mathcal{N}_{m}\mathcal{S}
;m\in\mathcal S\setminus\mathcal B_{\mathcal S},\ P\in \mathcal{N}_{m}\mathcal{S}\right\}$$
and if $\dim \mathcal B^{(0)}_{\mathcal S}=2$ and $\#\tilde{\mathcal B}^{(0)}_{\mathcal Z}\cap{\mathcal S}<\infty$, then, for a generic $P\in\mathbb P^3$, we have 
$$\mathfrak{N}_{\mathcal{S}}\cap \Sigma _{2}=\left\{ \mathcal{L}\in \mathfrak{N}
_{\mathcal{S}};P\in \mathcal{L}\right\} =\left\{\mathcal{N}_{m}\mathcal{S}
;m\in\mathcal S\setminus\tilde{\mathcal B}^{(0)}_{\mathcal S},\ P\in \mathcal{N}_{m}\mathcal{S}\right\}.$$
So, in any case, $a=c_{\nu }(\mathcal S)$ by definition
of the normal class of $\mathcal S$.\\
Now, for $b$, since $\#\mathcal B_{\mathcal S}<\infty$, we note that, for a generic projective plane $\mathcal H\subset\mathbb P^3$, we have
$$\mathfrak{N}_{\mathcal{S}}\cap \Sigma _{1,1}=\left\{ \mathcal{L}\in \mathfrak{
N}_{\mathcal{S}};\mathcal{L\subset H}\right\} =\left\{ \mathcal{N}_{m}
\mathcal{S};\ m\in\mathcal S\setminus \mathcal B_{\mathcal S},\ \mathcal{N}_{m}\mathcal{S}\subset \mathcal{H}\right\} .$$
We have 
$\mathcal{H}=V(a_1 X+a_2 Y+a_3 Z+a_4 T)\subset \mathbb{P}^{3} $
for some complex numbers $a_1$, $a_2$, $a_3$ and $a_4$.
Let $m[x:y:z:t]\in \mathbb P^3$. For a generic $\mathcal H$, we have
\begin{eqnarray*}
m\in\mathcal S\setminus\mathcal B_{\mathcal S},\ \mathcal{N}_{m}\mathcal{S}\subset \mathcal{H}
&\Leftrightarrow & 
m\in\mathcal S\setminus\mathcal B_{\mathcal S},\ \ m\in\mathcal H,\ \ 
   n_{\mathcal S}(m)\in\mathcal H\\
&\Leftrightarrow&\left\{ 
\begin{array}{c}
F(x,y,z,t)=0 \\ 
a_1 F_{x}+a_2 F_{y}+a_3 F_{z}=0\\ 
a_1 x+a_2 y+a_3 z+a_4 t=0
\end{array}
\right. .
\end{eqnarray*}
Hence $b=d_{\mathcal S}(d_{\mathcal S}-1)$.
Assume now that $\dim\mathcal B^{(0)}_{\mathcal S}=2$ with two dimensional part $V(H)$ 
and $\#\tilde{\mathcal B}^{(0)}_{\mathcal S}\cap{\mathcal S}<\infty$.
For a generic projective plane
$\mathcal H=V(A^\vee)\subset \mathbb P^3$, we have 
\begin{eqnarray*}
\mathfrak{N}_{\mathcal{S}}\cap \Sigma _{1,1}&=&\left\{ \mathcal{N}_{m}
\mathcal{S};\ m\in\mathcal S\setminus \tilde{\mathcal B}^{(0)}_{\mathcal S},\ 
\mathcal{N}_{m}\mathcal{S}\subset \mathcal{H}\right\}\\
&=&\{\mathcal N_m\mathcal S;\ m\in\mathcal S\setminus\tilde{\mathcal B}^{(0)}_{\mathcal S},\ \ m\in\mathcal H,\ \ 
   n_{\mathcal S}(m)\in\mathcal H\}.
\end{eqnarray*}
Now there are two cases:
\begin{itemize}
\item If $H$ divides $F_x$, $F_y$ and $F_z$ and then
$n_{\mathcal S}=[\frac{F_x}H:\frac{F_y}H:\frac{F_z}H]$ and
$b=d_{\mathcal S}(d_{\mathcal S}-d_H-1)$.
\item Otherwise $H=tH_1$, with $H_1$ dividing  $F_x$, $F_y$ and $F_z$ and 
$V(X)\subset V(x{F_y}-y{F_x},x{F_z}-z{F_x},y{F_z}-z{F_y})$. Hence 
$n_{\mathcal S}=[\frac{F_x}{H_1}:\frac{F_y}{H_1}:\frac{F_z}{H_1}]$. We have
\[
m\in\mathcal S\setminus(\mathcal H^\infty\cup\tilde{\mathcal B}^{(0)}_{\mathcal S}),\ \mathcal{N}_{m}\mathcal{S}\subset \mathcal{H}
\quad\Leftrightarrow\quad \left\{ 
\begin{array}{c}
F(x,y,z,t)=0,\ \ t\ne 0 \\ 
a_1 \frac{F_{x}}{H_1}+a_2\frac{ F_{y}}{H_1}+a_3\frac {F_{z}}{H_1}=0\\ 
a_1 x+a_2 y+a_3 z+a_4 t=0
\end{array}
\right.
\]
and
\[
m\in\mathcal S_\infty\setminus\tilde{\mathcal B}^{(0)}_{\mathcal S},\ \mathcal{N}_{m}\mathcal{S}\subset \mathcal{H}
\quad\Leftrightarrow\quad \left\{ 
\begin{array}{c}
F(x,y,z,t)=0\\
t=0\\ 
a_1 x+a_2 y+a_3 z=0
\end{array}
\right. ,
\]
so 
$$b=d_{\mathcal S}(d_{\mathcal S}-d_H)+d_{\mathcal S}-
\sum_{P\in \mathcal S\cap\mathcal H^\infty\cap\mathcal H}i_P\left(\mathcal S,V(a_1 \frac{F_{x}}{H_1}+a_2\frac{ F_{y}}{H_1}+a_3\frac {F_{z}}{H_1}),\mathcal H\right)$$
(due to the Bezout Theorem). Now let $P\in \mathcal S\cap\mathcal H^\infty\cap\mathcal H$,
we have $x\ne 0$ or $y\ne 0$ or $z\ne 0$. Assume for example that $x\ne 0$, we have
\begin{eqnarray*}
&\ &i_P\left(\mathcal S,V(a_1 \frac{F_{x}}{H_1}+a_2\frac{ F_{y}}{H_1}+a_3\frac {F_{z}}{H_1}),\mathcal H\right)=\\
&=&i_P\left(\mathcal S,V(t(-a_4\frac{F_x}{H_1}+a_2\frac{xF_y-yF_x}H+a_3\frac{xF_z-zF_x}H),\mathcal H\right)\\
&=&1 +i_P\left(\mathcal S,V(-a_4\frac{F_x}{H_1}+a_2\frac{xF_y-yF_x}H+a_3\frac{xF_z-zF_x}H),\mathcal H\right)\\
&=&2
\end{eqnarray*}
for a generic $\mathcal H$ and so
$b=d_{\mathcal S}(d_{\mathcal S}-d_H+1)$.
\end{itemize}
\end{proof}
Theorem \ref{formulegeneralesurface} follows from Theorem \ref{formulegeneralehypersurface}
and Proposition \ref{PROP1}.
\begin{proof}[Proof of Theorem \ref{factorisable}]
If $\dim\mathcal B^{(0)}_{\mathcal S}=2$, we saw in Proposition \ref{rqebasepoints}
that we can adapt our study to compute the degree of the reduced
normal polar $\tilde {\mathcal P}_{A,\mathcal S}$
associated to the rational map $\tilde\alpha_{\mathcal S}:\mathbb P^3\setminus\tilde{\mathcal B}^{(0)}_{\mathcal S}
\rightarrow\mathbb P^3$ such that $\boldsymbol{\alpha}_{\mathcal S}=H\cdot\boldsymbol{\tilde\alpha}_{\mathcal S}$. 
Using Proposition
\ref{rqebasepoints} and following the proof of Theorem 
\ref{formulegeneralehypersurface},
we obtain Theorem \ref{factorisable}.
\end{proof}
\section{Proof of Theorem \ref{thmsurfaces}}\label{sec:proofthm1}
We apply Theorem \ref{formulegeneralesurface}. 
Note that, since $\mathcal S$ is smooth, it has only a finite number of points of tangency with $\mathcal H_\infty$ (due to Zak's theorem on tangencies \cite[corolloray 1.8]{Zak}).
Since the surface is smooth, $\mathcal B_{\mathcal S}$ consists of points of tangency of 
$\mathcal S$ with $\mathcal H_\infty$ and of points of tangency of
$\mathcal S_\infty$ with $\mathcal U_\infty$. It remains to compute
the intersection multiplicity of $\mathcal S$ with a generic normal polar at these points.
Let us recall that if $A\not\in\mathcal H^\infty$, then
$$i_P(\mathcal S,\mathcal P_{A,\mathcal S})=\dim_\mathbb C \left(\left(
        \mathbb C[x,y,z,t]/I\right)_P \right)$$
where $I$ is the ideal $(F,E_{A,1,2}E_{A,1,3},E_{A,2,3})$ of $\mathbb C[x,y,z,t]$,
with the notation $E_{A,i,j}$ introduced in the proof of Proposition \ref{degrepolaire}.

To compute these quantities, it may be useful to make an appropriate change of coordinates with the use of a projective similitude of $\mathbb P^3$.
Note that:
\begin{itemize}
\item[*] The umbilical $\mathcal U_\infty$ is stable under the action
of the group of projective similitudes of $\mathbb P^3$.
\item[*] For any $P\in\mathcal H_\infty\setminus\mathcal U_\infty$,
there exists a projective similitude $\zeta$ of $\mathbb P^3$
mapping $[1:0:0:0]$ to $P$.
\footnote{Let $P[x_0:y_0:z_0:0]$ with $x_0^2+y_0^2+z_0^2=1$.
Assume for example that $z_0^2\ne 1$ (up to a permutation of the coordinates) and take $\zeta$
given by $\kappa'(b,A)$ (for any $b\in\mathbb C^n)$ with
$A=(u\, v\, w)$ where $u=(x_0,y_0,z_0)$ and $v=(x_0^2+y_0^2)^{-\frac 12}(y_0,-x_0,0)$ and $w=(x_0^2+y_0^2)^{-\frac 12}(x_0z_0,y_0z_0,-x_0^2-y_0^2)$.} 
\item[*] For any $P\in\mathcal U_\infty$,
there exists a projective similitude $\zeta$ of $\mathbb P^3$
mapping $[1:i:0:0]$ to $P$.
\footnote{Let $P[x_0:y_0:z_0:0]\in\mathcal U_\infty$.
Assume for example that $y_0\ne 0$ and $x_0^2+y_0^2\ne 0$
(up to a composition by a permutation matrix). A suitable $\zeta$ is given by $\kappa'(b,A)$ (for any $b\in\mathbb C^n)$ with
$A=\left(\begin{array}{ccc}\frac{x_0(y_0^2-1)}{2y_0^2}&-\frac{ix_0(1+y_0^2)}{2y_0^2}&\frac{\sqrt{x_0^2+y_0^2}}{y_0}\\
\frac{1+y_0^2}{2y_0^2}&\frac{i(1-y_0^2)}{y_0}&0\\
\frac{i(y_0^4+y_0^2x_0^2-y_0^2-x_0^2)}{y_0^2\sqrt{x_0^2+y_0^2}}&\frac{(1+y_0^2)\sqrt{x_0^2+y_0^2}}
    {2y_0^2}&\frac{ix_0}{y_0}\end{array}\right)$.} 

\end{itemize}
We recall that a multiple point of order $k$ of a plane curve
is ordinary if its tangent cone contains $k$ pairwise distinct lines
and that an ordinary cusp of a plane curve is a double point with a single
tangent line in the tangent cone, this tangent line being non contained in the cubic cone of the curve at this point.
\begin{itemize}
\item \underline{Let $P$ be a (non singular) point of tangency of $\mathcal S$ with
$\mathcal H_\infty$.}

We prove the following:
\begin{itemize}
\item[(a)] $i_P(\mathcal S,\mathcal P_{A,\mathcal S})=k^2$  for a generic $A\in\mathbb P^3 $  if
$P$ is an ordinary multiple point of order $k+1$ of $\mathcal S_\infty\setminus \mathcal U_\infty$.
\item[(b)] $i_P(\mathcal S,\mathcal P_{A,\mathcal S})=k(k+1)$ for a generic $A\in\mathbb P^3$ if 
$P$ is an ordinary multiple point of order $k+1$ of $\mathcal S_\infty$, which belongs to $\mathcal U_\infty$ and at which the tangent line to $\mathcal U_\infty$ is not contain in the tangent cone
of $\mathcal S_\infty$.
\item[(c)] $i_P(\mathcal S,\mathcal P_{A,\mathcal S})=3$ for a generic $A\in\mathbb P^3$ if $P$ is an ordinary cusp of $\mathcal S_\infty$, which belongs to $\mathcal U_\infty$ and at which the tangent line to $\mathcal U_\infty$ is not contain in the tangent cone
of $\mathcal S_\infty$.
\item[(d)] $i_P(\mathcal S,\mathcal P_{A,\mathcal S})=2$ for a generic $A\in\mathbb P^3$ if $P$ is an ordinary cusp of
$\mathcal S_\infty\setminus \mathcal U_\infty$.
\end{itemize}
Due to Lemma \ref{preservpolar}, 
we assume that $P[1:\theta:0:0]$ with $\theta=0$ (if $P\in\mathcal H_\infty\setminus\mathcal U_\infty$) or $\theta=i$ (if $P\in\mathcal U_\infty$). Since $\mathcal T_P\mathcal S=\mathcal H_\infty$, we suppose that $F_x(P)=F_y(P)=F_z(P)=0$
and $F_t(P)=1$ (without any loss of generality). 
Recall that the Hessian determinant $H_F$ of 
$F$ satisfies\footnote{see for example \cite{fredsoaz3}.}
$$H_F=\frac{(d_{\mathcal S}-1)^2}{x^2}\left|\begin{array}{cccc}
        0&F_{y}&F_{z}&F_{t}\\
        F_{y}&F_{yy}&F_{yz}&F_{yt}\\
        F_{z}&F_{yz}&F_{zz}&F_{zt}\\
        F_{t}&F_{yt}&F_{zt}&F_{tt}
\end{array}
        \right|.$$
Hence $H_F(P)\ne 0\ \Leftrightarrow\ [F_{yy}F_{zz}-F_{yz}^2](P)\ne 0$.
For a generic $A\in\mathbb P^3$, we have
$$\left(\frac{\mathbb C[x,y,z,t]}{I}\right)_P \cong
    \left(\frac{\mathbb C[x,y,z,t]}{(F,A_2,A_3)}\right)_P.$$
Recall that $A_2=atF_z-ctF_x+d(zF_x-xF_z)$
and
$A_3=-atF_y+btF_x+d(xF_y-yF_x)$
(with $A[a:b:c:d]$).
Using the Euler identity $xF_x+yF_y+zF_z+tF_t=d_{\mathcal S}F$,
we obtain that
$$\left(\frac{\mathbb C[x,y,z,t]}{I}\right)_P 
\cong
   \left(\frac{\mathbb C[x,y,z,t]}{(F,A'_2,A'_3)}\right)_P
\cong \left(\frac{\mathbb C[y,z,t]}{(F_*,A'_{2*},A'_{3*})}
  \right)_{(0,0,0)}$$
with 
$$A'_2:=atxF_z+ct(yF_y+zF_z+tF_t)-d(z(yF_y+zF_z+tF_t)+x^2F_z),
$$
$$A'_3:=-atxF_y-bt(yF_y+zF_z+tF_t)+d(x^2F_y+y(yF_y+zF_z+tF_t))$$
and with $G_{*}(y,z,t):=G(1,\theta+y,z,t)$  for any homogeneous $G$. 
In a neighbourhood of  $(0,0,0)$, $V(F_*)$ is given by $t=\varphi(y,z)$ with $\varphi(y,z)\in\mathbb C[[y,z]]$ and
\begin{equation}\label{deriveephi}
\varphi_y(y,z)=-\frac{F_y(1,\theta+y,z,\varphi(y,z))}{F_t(1,\theta+y,z,\varphi(y,z))}\quad\mbox{and}\quad\varphi_z(y,z)=-\frac{F_z(1,\theta+y,z,\varphi(y,z))}{F_t(1,\theta+y,z,\varphi(y,z))}.
\end{equation}
So
$$\left(\frac{\mathbb C[x,y,z,t]}{I}\right)_P
  \cong \frac{\mathbb C[[y,z]]}{(A'_{2**},A'_{3**})},$$
with 
$G_{**}(y,z):=G(1,\theta+y,z,\varphi(y,z))$. Now due to
\eqref{deriveephi}, we have
$$H:=A'_{2**}=(F_t)_{**} [-a\varphi\varphi_z+c\varphi(\varphi-(\theta+y)\varphi_y-z\varphi_z)+d(z((\theta+y)\varphi_y+z\varphi_z-\varphi)+\varphi_z)]$$
and
$$K:=A'_{3**}=(F_t)_{**} [a\varphi\varphi_y-b\varphi(\varphi-(\theta+y)\varphi_y-z\varphi_z)-d(\varphi_y+(\theta+y)((\theta+y)\varphi_y+z\varphi_z-\varphi))].$$
Hence
$$i_P(\mathcal S,\mathcal P_{A,\mathcal S})=i_{(0,0)}(\Gamma_H,\Gamma_K),$$
where $\Gamma_H$ and $\Gamma_K$ are the analytic plane curves
of respective equations $H$ and $K$ of $\mathbb C[[y,z]]$.

Note that $(H_y(0,0),H_z(0:0))=d(\varphi_{yz}(0,0),\varphi_{zz}(0,0))$. 
Analogously, we obtain $(K_y(0,0),K_z(0,0))=-d(1+\theta^2)(\varphi_{y,y}(0,0),\varphi_{yz}(0,0))$.
\begin{itemize}
\item[(a)] If $P\not\in\mathcal U_\infty$ and if
$P$ is an ordinary multiple point of order $k+1$ of $\mathcal S_\infty$, with our change of coordinates
we have $P[1:0:0:0]$ (i.e. $\theta=0$) and $V((\varphi_{k+1})_y)$
and $V((\varphi_{k+1})_z)$ have no common lines.\footnote{Recall that the tangent cone $V(\varphi_{k+1})$ of $\Gamma_{\varphi}$ at $(0,0)$
(corresponding to the tangent cone of $V(\mathcal S_\infty)$ at $P$)
has pairwise distinct tangent lines if and only if 
$V((\varphi_{k+1})_y)$
and $V((\varphi_{k+1})_z)$ have no common lines.
}
Then the first homogeneous parts of $H$ and $K$ have order $k$
and are  $H_k=d(\varphi_{k+1})_z$ and $K_k=-d(\varphi_{k+1})_y$
respectively.
Since $\Gamma_{H_k}$
and $\Gamma_{K_k}$ have no common lines, 
we conclude that $i_P(\mathcal S,\mathcal P_{A,\mathcal S})=k^2$.
\item[(b)]
Assume now that $P\in\mathcal U_\infty$ 
is an ordinary multiple point of $\mathcal S_\infty$, which belongs to $\mathcal U_\infty$ and at which the tangent line to $\mathcal U_\infty$ is not contain in the tangent cone
of $\mathcal S_\infty$.
With our changes of coordinates, 
this means that $P[1:i:0:0]$ (i.e. $\theta=i$) that 
$y$ does divide $\varphi_{k+1}$ (since $V(y)$ is the tangent line
to $\mathcal U_\infty$ at $P$) and that $V((\varphi_{k+1})_y)$
and $V((\varphi_{k+1})_z)$ have no common lines.

Note that the first homogeneous parts of $H$ and $K$ have respective orders $k$ and $k+1$ and are respectively
$H_k=d(\varphi_{k+1})_z$ and 
\begin{eqnarray*}
K_{k+1}&=&-di[2y(\varphi_{k+1})_y+z(\varphi_{k+1})_z-\varphi_{k+1}]\\
&=&-di\left[\left(2-\frac 1{k+1}\right) y(\varphi_{k+1})_y+\left(1-\frac 1{k+1}\right)z(\varphi_{k+1})_z\right],
\end{eqnarray*}
due to the Euler identity applied to $\varphi_{k+1}$.
Hence $i_P(\mathcal S,\mathcal P_{A,\mathcal S})=k(k+1)$ if
$V((\varphi_{k+1})_z)$ and $V(y(\varphi_{k+1})_y)$ have no common lines, which is true since $V((\varphi_{k+1})_z)$ and $V((\varphi_{k+1})_y)$ have no common lines and since $y$ does not divide $(\varphi_{k+1})_z$.
\item[(c)] Assume that $P\in\mathcal U_\infty$ is an ordinary cusp (of order 2)
of $\mathcal S_\infty$, which belongs to $\mathcal U_\infty$ and at which the tangent line to $\mathcal U_\infty$ is not contain in the tangent cone
of $\mathcal S_\infty$. With our changes of coordinates, 
this means that $P[1:i:0:0]$ (i.e. $\theta=i$), that 
$H_F(P)=0$ and $F_{zz}(P)\ne 0$ (since $V(y)$ is the tangent line
to $\mathcal U_\infty$ at $P$).

Note that $H=-d(F_{yz}(P)y+F_{zz}(P)z)+...$. In a neighbourhood of $(0,0)$, $H(y,z)=0\ \Leftrightarrow\ z=h(y)$
with

$h(y)=-\frac{F_{yz}(P)}{F_{zz}(P)}y-\frac{ F_{zzz}(P)F^2_{yz}(P)
  -2F_{yzz}(P)F_{yz}(P)F_{zz}(P)+F_{yyz}(P)F^2_{zz}(P)}{F_{zz}^3(P)
    }y^2+...$

and we obtain that $\val_yK(y,h(y))=3$ if
$P\not\in V(F_{yyy}F_{zz}^3-3F_{yyz}F_{zz}^2F_{yz}
+3F_{yzz}F_{zz}F_{yz}^2-F_{zzz}^3F_{yz}^3)$ which means that
the line $V(F_{yz}(P)y+F_{zz}(P)z)$ (corresponding to the
tangent cone of $V(F(1,y,z,0))$) is not contained
in the cubic cone of $V(F(1,y,z,0))$. Hence 
$i_P(\mathcal S,\mathcal P_{A,\mathcal S})=3$ if the node $P$ 
of $\mathcal S_\infty$ is ordinary.
\item[(d)] Assume now that $P$ is an ordinary cusp (of order 2) of
$\mathcal S_\infty\setminus \mathcal U_\infty$.

With our change of coordinates, this means that $P[1:0:0:0]$
(i.e. $\theta=0$), that $H_F(P)=0$ and  $P\not\in V(F_{yy},F_{yz},F_{zz})$. This implies that 
$F_{yy}(P)\ne 0$ or $F_{zz}(P)\ne 0$.

If $F_{yy}(P)\ne 0$, the tangent line to $\mathcal S_\infty$
at $P$ is given by $V(t,F_{yy}(P)y+F_{yz}(P)z)$.

If $F_{zz}(P)\ne 0$,  the tangent line to $\mathcal S_\infty$
at $P$ is given by $V(t,F_{zz}(P)z+F_{yz}(P)y)$.

The fact that the cusp is ordinary implies also that the tangent line
is not contained in the cubic cone of $V(F(1,y,z,0))$, i.e. this tangent line is not contained $V(F_{yyy}(P)y^3+3F_{yyz}(P)y^2z
+3F_{yzz}(P)yz^2+F_{zzz}(P)z^3](P)$.

Hence, we have either
$$F_{yy}[F_{yyy}F_{yz}^3-3F_{yyz}
F_{yz}^2F_{yy}+3F_{yzz}F_{yz}F_{yy}^2-F_{zzz}F_{yy}^3](P)\ne 0$$
or
$$F_{zz}[F_{zzz}F_{yz}^3-3F_{yzz}
F_{yz}^2F_{zz}+3F_{yyz}F_{yz}F_{zz}^2-F_{yyy}F_{zz}^3](P)\ne 0.$$
Note that, if $F_{yy}(P)$ and $F_{zz}(P)$ are both non null,
these two conditions are equivalent.

Assume for example that the first condition holds.
In a neighbourhood of $(0,0)$, $H(y,z)=0\,\Leftrightarrow\, y=h(z)$
and $K(y,z)=0\, \Leftrightarrow\, y=k(z)$, with
$$h'(z)=-\frac{H_z(h(z),z)}{H_y(h(z),z)}\quad\mbox{and}\quad
   k'(z)=-\frac{K_z(h(z),z)}{K_y(h(z),z)}.$$
Hence we have
$$\left(\frac{\mathbb C[x,y,z,t]}{I}\right)_P
  \cong \frac{\mathbb C[[y,z]]}{(y-h(z),y-k(z))} 
   \cong \frac{\mathbb C[[z]]}{((h-k)(z))} .$$
We have $h'(0)=k'(0)=-\varphi_{yz}(0,0)/\varphi_{yy}(0,0)$ and

$(h''-k'')(0)=\left[\varphi_{yy}^3\varphi_{zy}
\left(\varphi_{yyy}\varphi_{yz}^3-3\varphi_{yyz}\varphi_{yz}^2
   \varphi_{yy}+3\varphi_{yzz}\varphi_{yy}^2
   \varphi_{yz}-\varphi_{zzz}\varphi_{yy}^3\right)\right](0,0)$
\end{itemize}
Hence $i_P(\mathcal S,\mathcal P_{A,\mathcal S})=\dim_{\mathbb C}\frac{\mathbb C[[z]]}{((h-k)(z))}=2$.

\item \underline{Let $P$ be a simple (non singular) point of tangency of $\mathcal S_\infty$
with $\mathcal U_\infty$.} 

Let us prove that 
$i_P(\mathcal S,\mathcal P_{A,\mathcal S})=1$.
Due to Lemma \ref{preservpolar}, 
we can assume that $P=[1:i:0:0]$ (i.e. $\theta=i$) and that $F_t(P)= 1$.
As previously, we note that
$$\left(\frac{\mathbb C[x,y,z,t]}{I}\right)_P
  \cong \frac{\mathbb C[[y,z]]}{(A_{1**},A'_{3**})},$$
with
$$A_{1**}(y,z)=(F_t)_{**}[\varphi(b\varphi_z-c\varphi_y)-d((y+i)\varphi_z-z\varphi_y)]$$
and
$$A_{3**}=(F_t)_{**} [a\varphi\varphi_y-b\varphi(\varphi-(\theta+y)\varphi_y-z\varphi_z)-d(\varphi_y+(i+y)((i+y)\varphi_y+z\varphi_z-\varphi))].$$

The fact that $P$ is a simple contact point of $\mathcal S_\infty$
with $\mathcal U_\infty$ implies that $[F_x(P):F_y(P):F_z(P)]=[1:i:0]$ and that $\varphi_{zz}(0,0)\ne 1$.
Indeed $V(t,F(1,i+y,z,t))$
is given by $t=0,\, y=g(z)$ with $g(0)=0$ and $g'(z)=-\varphi_z(g(z),z)/\varphi_y(g(z),z)$ (in particular
$g'(0)=0$), so
$$\frac{\mathbb C[[y,z]]}{(y-g(z),1+(i+y)^2+z^2)}\cong
   \frac{\mathbb C[[z]]}{(1+(i+g(z))^2+z^2)}
$$
and finally
$$
i_P(\mathcal S_\infty,\mathcal U_\infty)=
     \val_z(1+(i+g(z))^2+z^2)=1+\val_z((i+g(z))g'(z)+z)
$$
which is equal to 2 if and only if  $\varphi_{zz}(0,0)\ne 1$.

In a neighbourhood of $(i,0)$, $A_{1**}$ can be rewritten
$\varphi-\kappa$ with $\kappa=d((y+i)\varphi_z-z\varphi_y)/(b\varphi_z-c\varphi_y)$.
Since $ \varphi_{zz}(0)\ne 1$,  $\kappa_z(0,0)\ne 0$
and, in a neighbourhood of $0$, $\varphi-\kappa=0$ corresponds to $y=h(z)$
with $h'(0)\ne 0$ (recall that $\varphi_y(0)=i\ne 0$ and that $A$ is generic) which gives
$$\left(\frac{\mathbb C[x,y,z,t]}{I}\right)_P
  \cong 
  \frac{\mathbb C[[y,z]]}{(y-h(z),A'_{3**})}$$
and finally $i_P(\mathcal S ,\mathcal P_{A,\mathcal S})=\val_z
A'_{3**}(h(z),z)=1$.
\end{itemize}
\section{Examples in $\mathbb P^3$}
\subsection{Normal class of quadrics}\label{secquadric}
The aim of the present section is the study of the normal class
of every irreducible quadric. Let $\mathcal S=V(F)\subset\mathbb P^3$ be an irreducible quadric.
We recall that, up to a composition by
$\varphi\in \widehat{Sim_{\mathbb C}(3)}$, one can suppose that $F$
has the one of the following forms:
\begin{eqnarray*}
&(a)& F(x,y,z,t)=x^2+\alpha y^2+\beta z^2 +t^2\\
&(b)& F(x,y,z,t)=x^2+\alpha y^2+\beta z^2\\
&(c)& F(x,y,z,t)=x^2+\alpha y^2-2tz\\
&(d)& F(x,y,z,t)=x^2+\alpha y^2+t^2,
\end{eqnarray*}
with $\alpha,\beta$ two non zero complex numbers.
Spheres, ellipsoids and hyperboloids are particular cases of (a),
paraboloids (including the saddle surface) are particular cases of (c), (b) correspond to cones and (d) to cylinders.

We will see, in Appendix \ref{cylindreetrevolution}, that in the case (d)
(cylinders) and in the cases (a) and (b) with $\alpha=\beta=1$, the normal class of the quadric is naturally related to the normal class of a conic.
\begin{prop}\label{quadric}
The normal class of a sphere is 2.

The normal class of a quadric $V(F)$ with $F$ given by (a) is 6 if $1,\alpha,\beta$ are pairwise distinct.

The normal class of a quadric $V(F)$ with $F$ given by (a) is 4 if $\alpha=1\ne\beta$.

The normal class of a quadric $V(F)$ with $F$ given by (b) is 4 if $1,\alpha,\beta$ are pairwise distinct.

The normal class of a quadric $V(F)$ with $F$ given by (b) is 2 if $\alpha=1\ne\beta$.

The normal class of a quadric $V(F)$ with $F$ given by (b) is 0 if $\alpha=\beta=1$.

The normal class of a quadric $V(F)$ with $F$ given by (c) is 5 if $\alpha\ne 1$ and 3 if $\alpha=1$.

The normal class of a quadric $V(F)$ with $F$ given by (d) is 4 if
$\alpha\ne 1$ and 2 if $\alpha=1$.
\end{prop} 
\begin{coro}
The normal class of the saddle surface $\mathcal S_1=V(xy-zt)$ is 5.

The normal class of the ellipsoid $\mathcal E_1=V(x^2+2y^2+4z^2-t^2)$ with three different length of axis is 6.

The normal class of the ellipsoid $\mathcal E_2=V(x^2+4y^2+4z^2-t^2)$ with two different length of axis is 4.
\end{coro}
\begin{proof}[Proof of Proposition \ref{quadric}]
Let $\mathcal S=V(F)$ be a quadric with $F$ of the form (a), (b), (c) or (d).
\begin{itemize}
\item The easiest cases is (a) with $1,\alpha,\beta$ pairwise distinct since $\mathcal B_{\mathcal S}$ is empty. In this case, since the generic
degree of the normal polar curves is 3 and since $\mathcal E_1$ has degree 2, we simply have $c_{\nu}(\mathcal E_1)=2\cdot 3=6$ (due to Theorem \ref{formulegeneralesurface}).
\item 
The case of a sphere $\mathcal S$ is analogous. In this case, $\tilde {\mathcal B}^{(0)}_{\mathcal S}\cap\mathcal S=\emptyset$ and $\deg\tilde{\mathcal P}_{A,\mathcal S}=1$ for a generic $A\in\mathbb P^3$ (see Example \ref{exemple1}).
Hence, we have $c_{\nu}(\mathcal E_1)=2\cdot 1=2$ (due to Theorem 
\ref{factorisable}).
\item In case (a) with $\alpha=1\ne\beta$, the set $\mathcal B_{\mathcal S}$ contains two points $[1:\pm i:0:0]$.
We find the parametrization $\psi(y)=[1:\pm i+y:0:0]$ of $\mathcal P_{A,\mathcal S}$ at the neighbourhood of $P[1:\pm i:0:0]$,
which gives $i_P(\mathcal S,\mathcal P_{A,\mathcal S})=\val_z(1+(\pm i+y)^2)=1$ and so $c_{\nu}(\mathcal S)=2\cdot 3-1-1=4$.
\item In case (b) with $\alpha$, $\beta$ and $1$ are pairwise distinct,
the set $\mathcal B_{\mathcal S}$ contains a single point $P[0:0:0:1]$
and a parametrization of ${\mathcal P}_{A,\mathcal S}$ 
in a neighbourhood of $P$ is
\begin{equation}\label{parametrisation1}
\psi(x)=\left[x:-\frac{bx}{d(\alpha-1)x-a}:\frac{cx}{a+d(1-\beta)x}:1\right].
\end{equation}
Hence $i_P(\mathcal S,\mathcal P_{A,\mathcal S})=\val_x(F(\psi(x)))=2$
and so $c_\nu(\mathcal S)=2\cdot 3-2=4$.
\item In case (b) with $\alpha=1\ne\beta$, we have
$\mathcal B_{\mathcal S}=\{P,P'_+,P'_-\}$ with $P[0:0:0:1]$
and $P'_{\pm}[1:\pm i:0:0]$. A
parametrization of ${\mathcal P}_{A,\mathcal S}$ 
in a neighbourhood of $P$ is given by \eqref{parametrisation1} with
$\alpha=1$ and so $i_P(\mathcal S,\mathcal P_{A,\mathcal S})=2$.
A parametrization of ${\mathcal P}_{A,\mathcal S}$ at a neighbourhood of $P'_{\pm}$ is  $\psi(z)=[1:\pm i+y:0:0]$ and so 
$i_{P'_{\pm}}(\mathcal S,\mathcal P_{A,\mathcal S})=1$. Hence
 $c_\nu(\mathcal S)=2\cdot 3-2-1-1=2$.
\item In case (b) with $\alpha=\beta=1$, for a generic $A\in\mathbb P^3$, we have $\deg\tilde{\mathcal P}_{A,\mathcal S}=1$ (see Example \ref{exemple1}) but here $\tilde{\mathcal B}^{(0)}_{\mathcal S}\cap\mathcal S=\{[0:0:0:1]\}$. We find
the parametrization $\psi(x)=[x:(bx/a):(cx/a):1]$ of $\tilde{\mathcal P}_{A,\mathcal S}$ at the neighbourhood of $P[0:0:0:1]$. Hence $i_P(\mathcal S,\tilde{\mathcal P}_{A,\mathcal S})=2$ and so $c_{\nu}(\mathcal S)=2\cdot 1-2=0$.
\item In case (c) with $\alpha\ne 1$, the only point of $\mathcal B_{\mathcal S}$
is $P_1[0:0:1:0]$ and a parametrization of ${\mathcal P}_{A,\mathcal S}$
at the neighbourhood of this point is
\begin{equation}\label{parametrisation2}
\psi(t)=\left[\frac{at^2}{d+(d-c)t}:\frac{bt^2}{t(d-c\alpha)+\alpha d}:1:t\right],
\end{equation}
which gives $i_{P_1}(\mathcal S,\mathcal P_{A,\mathcal S})=1$. Hence
$c_{\nu}(\mathcal S)=2\times 3-1=5$.
\item In case (c) with $\alpha=1$, $\mathcal B_{\mathcal S}$ is made of
three points: $P_1[0:0:1:0]$, $P_{2,\pm}[1:\pm i:0:0]$.
As in the previous case, a parametrization of ${\mathcal P}_{A,\mathcal S}$
at the neighbourhood of $P_1$ is given by \eqref{parametrisation2}
with $\alpha=1$
and so $i_{P_1}(\mathcal S,\mathcal P_{A,\mathcal S})=1$.
Now, a parametrization of ${\mathcal P}_{A,\mathcal S}$
at the neighbourhood of $P_{2,\pm}$ is $\psi(t)=[1:\pm i+y:0:0]$
and so $i_{P_{2,\pm}}=(\mathcal S,\mathcal P_{A,\mathcal S})=\val_y
(1+(y\pm i)^2)=1$.
\item For the case (d), due to Proposition \ref{propcylinder},
$c_\nu(\mathcal S)=c_\nu(\mathcal C)$ with $\mathcal C=V(x^2+\alpha y^2+z^2)\subset\mathbb P^2$ which is a circle if $\alpha=1$ and an ellipse
otherwise. Hence, due to Theorem \ref{thmcurves}, $c_\nu(\mathcal C)=2+2-0-1-1=2$ if $\alpha=1$ and $c_\nu(\mathcal C)=2+2=4$ otherwise.
\end{itemize}
\end{proof}
\subsection{Normal class of a cubic surface with singularity $E_6$}\label{cubic}
Consider $S=V(F)\subset\mathbb{P}^{3}$ with $F(x,y,z,t):=x^{2}z+z^{2}t+y^{3}$.
$\mathcal{S}$ is a singular cubic surface with $E_{6}-$singularity at
$p[0:0:0:1]$. Let a generic $A[a:b:c:d]\in\mathbb P^3$. 
The ideal of the normal polar $\mathcal{P}_{\mathcal{S},A}$ is given by
$I(\mathcal{P}_{\mathcal{S},A})=\langle H_1,H_2,H_3\rangle\subset\mathbb{C}[x,y,z,t]$ with
$H_1:=(y(x^{2}+2zt)-3y^{2}z)d-b(x^{2}+2zt)t+3y^{2}ct$,
$H_2:=(x(x^{2}+2zt)-2xz^{2})d-a(x^{2}+2zt)t+2xztc$ and $H_3:=(-2xzy+3xy^{2}%
)d-3ay^{2}t+2xztb$. $\mathcal B_{\mathcal S}$ is made of 
two points: $p$ and $q[0:0:1:0]$. Actually $q$ is the
point of tangency of $\mathcal{S}$ with $\mathcal{H}_{\infty}$. This point is  an ordinary cusp of $\mathcal{S}_{\infty}$.
\begin{enumerate}
\item Study at $p$.

Near p the ideal of the normal polar, in the chart $t=1$, $H_3=0$
gives $z=g(x,y):=\frac{3y^{2}(-xd+a)}{2x(-yd+b)}$. Now $V(A_1(x,y,g(x,y),1))$
corresponds to a quintic  with a cusp at the
origine (and with tangent cone $V(y^{2}))$. Its single branch has Puiseux expansion
$y^2=-\frac b{3a}x^3+o(x^3)$, with probranches $y=\varphi_{\varepsilon}(x)$
with $\varphi_\varepsilon(x)=i\varepsilon\sqrt{\frac b{3a}}x^{\frac 32}+o(x^{\frac 32})$
for $\varepsilon\in\{\pm 1\}$. Hence, $g(x,\varphi_\varepsilon(x))=-\frac {x^2}2+o(x^2)$.
Hence parametrizations of the probranches of $\mathcal P_{A,\mathcal S}$ at a neighbourhood
of $p$ are
$$\Gamma_{\varepsilon}(x)=[x:\varphi_\varepsilon(x):g(x,\varphi_\varepsilon(x)):1]$$
and $F(\Gamma_{\varepsilon}(x))=-\frac {x^4}4+o(x^{4})$.
Therefore $i_{p}(\mathcal{P}_{\mathcal{S},A},\mathcal S)=8$.
\item Study at $q.$ 

Assume that $b=1$. Near $q[0:0:1:0]$, in the chart $z=1$, $H_3=0$ gives
$t=h(x,y):=\frac{d(-2+3y)xy}{3ay^{2}-2x}$ and
$V(H_2(x,y,1,h(x,y)))$ is a quartic with a (tacnode) double point in $(0,0)$ with
vertical tangent and which has Puiseux expansion
$$x=\theta_{\varepsilon}(y)=\omega_{\varepsilon}a\, y^{2}+o(y^{2}),$$
with $\omega_{\varepsilon}=\frac{3-d}{2}+\frac{\varepsilon}{2}\sqrt{d(d-6)}$
for $\varepsilon\in\{\pm 1\}$ and
 $h(\theta_\varepsilon(y),y)=-\frac{2d\omega_\varepsilon}
  {3-2\omega_\varepsilon}y+o(y)$.
Hence parametrizations of the probranches of $\mathcal{P}_{\mathcal{S},A}$ in a
neighbourhood of $q$ are given by
$$\Gamma_{\varepsilon}(y):=[\theta_{\varepsilon}(y):y:1:h(\theta_\varepsilon(y),y)]$$
for $\varepsilon\in\{\pm 1\}$ and 
$F(\Gamma_{\varepsilon}(y))=-\frac{2\omega_\varepsilon d}{3-2\omega_\varepsilon}y+o(y)$.
Hence $i_{q}(\mathcal{P}_{\mathcal{S},A},\mathcal S)=2$

We can also apply directly Item (c) of Section \ref{sec:proofthm1}
to prove that $i_{q}(\mathcal{P}_{\mathcal{S},A},\mathcal S)=2$.
\end{enumerate}
Therefore, due to Theorem \ref{formulegeneralesurface}, the normal class of 
$\mathcal S=V(x^{2}z+z^{2}t+y^{3})\subset \mathbb P^3(\mathbb C)$
is
$$c_{\nu}(\mathcal{S})=3\cdot(3^{2}-3+1)-8-2=11.$$
\section{Normal class of plane curves~: Proof of Theorem \ref{thmcurves}}\label{proofcurve}
Let $\mathbf V$ be a three dimensional complex vector space and set $\mathbb P^2:=\mathbb P(\mathbf V)$ with projective coordinates
$x,y,z$. We denote by $\ell_\infty=V(z)$ the line at infinity.

Let $\mathcal C=V(F)\subset \mathbb P^2$
be an irreducible curve of degree $d\ge 2$.
For any nonsingular $m[x:y:z]\in\mathcal C$ (with coordinates $\mathbf m=(x,y,z)\in\mathbb C^3$), we write $\mathcal T_m\mathcal C$
for the tangent line to $\mathcal C$ at $m$. 
If $\mathcal T_m\mathcal C\ne\ell_\infty$, then $n_{\mathcal C}(m)=[F_x:F_y:0]$
is well defined in $\mathbb P^2$ and
{\bf the projective normal line} $\mathcal N_m\mathcal C$ to $\mathcal C$ at $m$ is the line
$(m\, n_{\mathcal C}(m))$ if $n_{\mathcal C}(m)\ne m$.
An equation of this normal line is then given by $\langle \mathbf {N}_\mathcal C({m}),\cdot\rangle$ where $N_\mathcal C:\mathbb P^2\dashrightarrow\mathbb P^2$ is the rational map defined by
\begin{equation}\label{NC}
\mathbf{N}_\mathcal C(\mathbf{m}):=\mathbf{m}\wedge\left(\begin{array}{c}F_x\\F_y\\0\end{array}\right)
   =\left(\begin{array}{c}-zF_y(m)\\ zF_x(m)\\ xF_y(m)-yF_x(m)\end{array}\right).
\end{equation}
\begin{lem}
The base points of $(N_\mathcal C)_{|\mathcal C}$ are the singular points of $\mathcal C$,
the points of tangency with the line at infinity and the points of $\{I,J\}\cap\mathcal C$.
\end{lem}
\begin{proof}
A point $m\in\mathcal C$ is a base point of $N_\mathcal C$ if and only if $F_x=F_y=0$ or $z=xF_y-yF_x=0$. 
Hence, singular points of $\mathcal C$ are base points of $N_{\mathcal C}$.

Let $m=[x:y:z]$ be a nonsingular point of $\mathcal C$. 
First $F_x=F_y=0$ is equivalent to $\mathcal T_m\mathcal C=\ell_\infty$.
Assume now that $z=xF_y-yF_x=0$ and $(F_x,F_y)\ne(0,0)$. Then $m=[x:y:0]=[F_x:F_y:0]$ and, 
due to the Euler formula, we have $0=-zF_z=xF_x+yF_y$ and so $x^2+y^2=0$, which implies $m=I$ or $m=J$.

Finally note that if $m\in\{I,J\}\cap\mathcal C$, then $m=[-y:x:0]$ and, 
due to the Euler formula, $0=-zF_z=xF_x+yF_y=xF_y-yF_x$.
\end{proof}
Since the degree of each non zero coordinate of $\mathbf
\mathbf N_\mathcal C$ is $d$, we have
\begin{equation}\label{lemmefondamental}
c_\nu(\mathcal C)
   =d^2-\sum_{P\in \mathcal \base\left({(N_\mathcal C)}_{\vert\mathcal C}\right)} i_P(\mathcal C,V(\langle L,\mathbf N_{\mathcal C}(\cdot)\rangle)),
\end{equation}
for a generic $L\in\mathbb P^2$, where we write 
$\base\left({(N_\mathcal C)}_{\vert\mathcal C}\right)$ 
for the set of base points of 
${(N_\mathcal C)}_{\vert\mathcal C}$.
The set $V(\langle L,\mathbf N_{\mathcal C}(\cdot)\rangle)\subset \mathbb P^2$ is called 
{\bf the normal polar} of $\mathcal C$ with respect to $L$. It satisfies
$$m \in V(\langle L, \mathbf N_{\mathcal C}(\cdot)\rangle)\quad \Leftrightarrow\quad  \mathbf N_{\mathcal C}(\mathbf m)=0\ \mbox{or}\ L\in
\mathcal N_m(\mathcal C).$$
Now, to compute the generic intersection numbers, we use the notion of probranches \cite{Halphen,Wall,Wall2}. 
See section 4 of
\cite{fredsoaz1} for details.
Let $P\in\mathcal C$ be an indeterminancy point of $N_{\mathcal C}$ and let us write $\mu_P$ for the multiplicity of $\mathcal C$ at $P$.
Recall that $\mu_P=1$ means that $P$ is a nonsingular point of $\mathcal C$. 
Let $M\in GL(\mathbf V)$ be such that $M(\mathbf O)=\mathbf P$ with $\mathbf O=(0,0,1)$ (we set also $O=[0:0:1]$)
and such that $V(x)$ is not contained in the tangent cone of $V(F\circ M)$ at $O$.
Recall that the equation of this tangent cone is the homogeneous part of lowest degree in $(x,y)$ of $F(x,y,1)\in\mathbb C[x,y]$ and
that this lowest degree is $\mu_P$.
Using the combination of the Weierstrass preparation theorem and of the Puiseux expansions, 
$$F\circ M(x,y,1)=U(x,y)\prod_{j=1}^{\mu_P}
(y-g_j(x)),$$ 
for some $U(x,y)$ in the ring of convergent series in $x,y$ with $U(0,0)\ne 0$ and where
$g_j(x)=\sum_{m\ge 1}a_{j,m}x^{\frac m{q_j}}$ for some integer $q_j\ne 0$.
The $y=g_j(x)$ correspond to the equations of the probranches of $\mathcal C$ at $P$.
Since $V(x)$ is not contained in the tangent cone of $V(F\circ M)$ at $O$, the valuation in $x$
of $g_j$ is strictly larger than or equal to 1 and so the probranch $y=g_j(x)$ is tangent to $V(y-xg_j'(0))$.
We write $\mathcal T_P^{(i)}:=M(V(y-xg_j'(0)))$ the associated (eventually singular) tangent line to $\mathcal C$ at $P$
($\mathcal T_P^{(i)}$ is the tangent to the branch of $\mathcal C$ at $P$ corresponding to this probranch)
and we denote by $i_P^{(j)}$ the tangential intersection number of this probranch:
$$i_P^{(j)}=\val_x (g_j(x)-xg_j'(0))=\val_x (g_j(x)-xg_j'(x)).$$
We recall that for any homogeneous polynomial $H\in\mathbb C[x,y,z]$, we have
\begin{eqnarray*}
i_P(\mathcal C,V(H))&=& i_{O}(V(F\circ M),V(H\circ M))\\
&=&\sum_{j=1}^{\mu_P} \val_x(H(M(G_j(x)))),
\end{eqnarray*}
where $G_j(x):=(x,g_j(x),1)$.
With these notations and results, we have
$$\Omega(\mathcal C,\ell_\infty)=\sum_{P\in\mathcal C\cap\ell_\infty} (i_P(\mathcal C,\ell_\infty)-\mu_P(\mathcal C))
= \sum_{P\in\mathcal C\cap\ell_\infty} \sum_{j:\mathcal T_P^{(j)}=\ell_\infty} (i_P^{(j)}-1).$$
For a generic $L\in\mathbf V^\vee$, we also have
\begin{eqnarray*}
i_P(\mathcal C,V(L\circ N_{\mathcal C}))
&=&\sum_{j=1}^{\mu_P} \val_x(L(N_{\mathcal C}(M(G_j(x)))))\\
&=&\sum_{j=1}^{\mu_P} \min_k \val_x([N_{\mathcal C}\circ M]_k(G_j(x))),
\end{eqnarray*}
where $[\cdot ]_k$ denotes the $k$-th coordinate.
Moreover, due to \eqref{NC}, as seen in Proposition 16 of \cite{fredsoaz2}, we have
$$\mathbf{N}_{\mathcal C}\circ M(\mathbf m)= Com(M)\cdot (\mathbf{m}\wedge\left[\Delta_{\mathbf{A}} 
       G(\mathbf m)
     \cdot \mathbf{A}+\Delta_{\mathbf{B}} G(\mathbf m)\cdot \mathbf{B}\right]),$$
where $G:=F\circ M$, $\mathbf{A}:=M^{-1}(1,0,0)$, $\mathbf{B}:=M^{-1}(0,1,0)$ and $\Delta_{(x_1,y_1,z_1)}H=x_1H_x+y_1H_y+z_1H_z$.
As seen in Lemma 33 of \cite{fredsoaz1}, we have
$$\Delta_{(x_1,y_1,z_1)} G(x,g_j(x),1)=R_j(x)W_{(x_1,y_1,z_1),j}(x),$$
where
$R_j(x)=U(x,g_j(x))\prod_{j'\ne j}(g_{j'}(x)-g_j(x))$ and $W_{(x_1,y_1,z_1),j}(x):=y_1-x_1g'_j(x)+z_1(xg'_j(x)-g_j(x))$.
Therefore, for a generic $L\in\mathbf{V}^\vee$, we have
$$
i_P(\mathcal C,V(L\circ N_{\mathcal C}))=
V_P+\sum_{j=1}^{\mu_P} \min_k \val_x([G_j(x)\wedge (W_{\mathbf{A},j}(x)\cdot \mathbf A
     +W_{\mathbf{B},j}(x)\cdot \mathbf B)]_k)
$$
where $V_P:=\sum_{j=1}^{\mu_P}\sum_{j'\ne j}\val(g_{j'}-g_j)$. 
Now, we write $h_P^{(j)}:=\min_k \val_x([G_j(x)\wedge (W_{\mathbf{A},j}(x)\cdot \mathbf A
     +W_{\mathbf{B},j}(x)\cdot \mathbf B)]_k)$ and $h_P:=\sum_{j=1}^{\mu_P}h_P^{(j)}$.
Note that $V(P)=0$ if $P$ is a nonsingular point of $\mathcal C$.
We recall that, due to Corollary 31 of \cite{fredsoaz1}, we have
$$\sum_{P\in \mathcal C\cap \base(N_{\mathcal C})} V_P= d(d-1)-d^\vee$$
and so, due to \eqref{lemmefondamental}, we obtain
\begin{equation}\label{lemme fondamental2}
c_\nu(\mathcal C)
   =d+d^\vee -\sum_{P\in \mathcal C\cap\base(N_\mathcal C)} h_P.
\end{equation}
Now we have to compute the contribution $h_P^{(j)}$ of each probranch of each 
$P\in\mathcal C\cap \base(N_{\mathcal C})$. We have seen, in Proposition 29 of \cite{fredsoaz1},
that we can adapt our choice of $M$ to each probranch (or, to be more precise, to each branch corresponding to the probranch).
This fact will be useful in the sequel. In particular, for each probranch, we take $M$ such that $g_j'(0)=0$ so
$G_j(x)\wedge (W_{\mathbf{A},j}(x)\cdot \mathbf A +W_{\mathbf{B},j}(x)\cdot \mathbf B)$ can be rewritten:
\begin{equation}\label{probranche1}
\left(\begin{array}{c}x\\ g_j(x)\\1\end{array}\right)\wedge\left(\begin{array}{c}x_Ay_A-(x_A^2+x_B^2)g'_j(x)+x_By_B+(z_Ax_A+z_Bx_B)(xg_j'(x)-g_j(x))\\
       y_A^2+y_B^2-(x_Ay_A+x_By_B)g'_j(x)+(z_Ay_A+z_By_B)(xg_j'(x)-g_j(x))\\
       y_Az_A+y_Bz_B-(x_Az_A+x_Bz_B)g'_j(x)+(z_A^2+z_B^2)(xg_j'(x)-g_j(x))\end{array}\right) .
\end{equation}
\begin{itemize}
\item Assume first that $P$ is a point of $\mathcal C$ outside $\ell_\infty$. Then for $M$
as above and such that $z_A=z_B=0$, we have
$$ G_j(0)\wedge (W_{\mathbf{A},j}(0)\cdot \mathbf A
     +W_{\mathbf{B},j}(0)\cdot \mathbf B)=\left(\begin{array}{c}
   -y_A^2-y_B^2\\x_Ay_A+x_By_B\\0\end{array}\right)$$
which is non null since $(y_A,y_B)\ne(0,0)$ and since $\mathbf A$ and $\mathbf B$ are linearly independent.
So $h_P^{(j)}=0$.
\item Assume now that $P\in\mathcal C\cap \ell_\infty\setminus\{I,J\}$ and $\mathcal T_P^{(j)}\ne\ell_\infty$. Then $y_A+iy_B\ne 0$
and $y_A-iy_B\ne 0$ (since $I,J\not\in \mathcal T_P^{(j)}$) and so $y_A^2+y_B^2\ne 0$ which together with
\eqref{probranche1} implies that $h_P^{(j)}=0$ as in the previous case.
\item Assume that $P\in\mathcal C\cap \ell_\infty\setminus\{I,J\}$ and $\mathcal T_P^{(i)}=\ell_\infty$.
Assume that $M(1,0,0)=(1,i,0)$. Hence $\mathbf A+i\mathbf B=(1,0,0)$. Then $y_A=y_B=0$, $x_A+ix_B=1$, $z_A+iz_B=0$. So $z_A^2+z_B^2=0$
and $z_Ax_A+z_Bx_B=z_A\ne 0$ (since $z_B=iz_A$ and $x_B=i(x_A-1)$). 
Note that $P\ne J$ implies also that $x_A-ix_B\ne 0$. So that $x_A^2+x_B^2\ne 0$. 
Hence, due to \refeq{probranche1}, $G_j(x)\wedge (W_{\mathbf{A},j}(x)\cdot \mathbf A +W_{\mathbf{B},j}(x)\cdot \mathbf B)$
is equal to
$$
\left(\begin{array}{c}x\\ g_j(x)\\1\end{array}\right)\wedge\left(\begin{array}{c}(x_A^2+x_B^2)g'_j(x)+z_A(xg_j'(x)-g_j(x))\\
      0\\
       z_Ag_j'(x)\end{array}\right) .
$$
Therefore we have $h_P^{(j)}=\val_x((x_A^2+x_B^2)g'_j(x))=i_P^{(j)}-1$.
\item Assume that $P=I$ and that $\mathcal T_P^{(j)}=\ell_\infty$. Take $M$ such that $M(\mathbf O)=(1,i,0)$, $\mathbf B=(1,0,0)$ and so $\mathbf A=(-i,0,1)$. Due to \refeq{probranche1}, 
$G_j(x)\wedge (W_{\mathbf{A},j}(x)\cdot \mathbf A +W_{\mathbf{B},j}(x)\cdot \mathbf B)$
is equal to
\begin{equation}
\left(\begin{array}{c}x\\ g_j(x)\\1\end{array}\right)\wedge\left(\begin{array}{c}-i(xg_j'(x)-g_j(x))\\
      0\\
       ig_j'(x)+(xg_j'(x)-g_j(x))\end{array}\right) .
\end{equation}
Note that each coordinate has valuation at least equal to $i_P^{(j)}=\val g_j$ and that the term of
degree $i_P^{(j)}$ of the second coordinate is the term of degree $i_P^{(j)}$ of 
$$-i(xg'_j(x)-g_j(x))+xig'_j(x)=ig_j(x)\ne 0$$
which is non null. Therefore $h_P^{(j)}=i_P^{(j)}$.
\item Assume finally that $P=I$ and that $\mathcal T_P^{(j)}\ne\ell_\infty$. Take $M$ such that $M(\mathbf O)=(1,i,0)$, 
$\mathbf B=(0,1,0)$ and so $\mathbf A=(0,-i,1)$.
Due to \refeq{probranche1}, 
$G_j(x)\wedge (W_{\mathbf{A},j}(x)\cdot \mathbf A +W_{\mathbf{B},j}(x)\cdot \mathbf B)$
is equal to
\begin{equation}
\left(\begin{array}{c}x\\ g_j(x)\\1\end{array}\right)\wedge\left(\begin{array}{c}0\\-i(xg_j'(x)-g_j(x))\\
      -i+(xg_j'(x)-g_j(x))\end{array}\right) .
\end{equation}
Note that each coordinate has valuation at least equal to $1$ and that the term of
degree $1$ of the second coordinate is $ix\ne 0$. Hence $i_P^{(j)}=1$.
\end{itemize}
Note that the case $P=J$ can be treated in the same way than the case $P=I$.

Theorem \ref{thmcurves} follows from \eqref{lemme fondamental2} and from the previous
computation of $h_P$.

\begin{appendix}
\section{Dimension decrease}\label{cylindreetrevolution}
Here, we consider two particular cases of hypersurfaces the normal class of which is equal to the normal class of a hypersurface of lower dimension: cylinders
(i.e. cones at a point at infinity) and revolution hypersurfaces
(circles fibers).

Let $n\ge 3$.
Let $\tilde F\in\mathbb C[u_1,...,u_n]$ be homogeneous.
We call {\bf cylinder of base $\tilde{\mathcal Z}=V(\tilde F)\subset\mathbb P^n$
and of axis $V(x_2,...,x_n)\subset\mathbb P^n$} the hypersurface $V(F)\subset\mathbb P^n$, with
$F(x_1,\dots,x_{n+1}):=\tilde F(x_2,\dots,x_{n+1})$.

\begin{prop}\label{propcylinder}
Let $n\ge 3$ and $d\ge 2$.
Let $\mathcal Z=V(F)\subset\mathbb P^n$ be the cylinder
of axis $V(x_2,...,x_n)\subset\mathbb P^n$
and of base $\tilde{\mathcal Z}=V(\tilde F)\subset\mathbb P^{n-1}$.
Then $c_{\nu}(\mathcal Z)=c_{\nu}(\tilde{\mathcal Z})$.
\end{prop}
\begin{proof}
Note that $\mathcal Z\cap V(x_2,...,x_{n+1})\subset\sing(\mathcal Z)\subset\mathcal B_{\mathcal Z}$.
Let $m[x_1^{(1)}:\cdots:x_{n+1}^{(1)}]\in\mathcal Z\setminus V(x_2,...,x_{n+1})$
and $P[x_1^{(0)}:\cdots:x_{n+1}^{(0)}]\in\mathbb P^n\setminus V(x_2,...,x_{n+1})$.
Set $\tilde m[x_2^{(1)}:\cdots:x_{n+1}^{(1)}]\in\tilde{\mathcal Z}$
and $\tilde P[x_2^{(0)}:\cdots:x_{n+1}^{(0)}]\in\mathbb P^{n-1}$.
Note that
$n_{\mathcal Z}(m)[0:\tilde F_{u_1}(\tilde{\mathbf m}):\cdots:\tilde F_{u_{n-1}}(\tilde{\mathbf m}):0]\in\mathbb P^{n}$.
\begin{itemize}
\item Let $\mathcal H=V(\alpha x_1+\beta x_{n+1})\subset \mathbb P^n$ be a hyperplane orthogonal to $V(x_2,...,x_n)$ such that $\mathcal H\ne\mathcal H^\infty$ (i.e. $\alpha\ne 0$). Assume $m\in\mathcal H$.
Then $m\in\mathcal B_{\mathcal Z}\, \Leftrightarrow\, \tilde m\in\mathcal B_{\tilde{\mathcal Z}}$.
If $m\in\mathcal H\cap\mathcal Z\setminus\mathcal B_{\mathcal Z}$, then
$\mathcal N_m(\mathcal Z)\subset \mathcal H$.
\item Assume $P\in \mathbb P^n\setminus V(x_1,x_{n+1})$. Then
$\mathcal H:=V(x_1^{(0)}x_{n+1}-x_{n+1}^{(0)}x_{1})$ is the unique
hyperplane orthogonal to $V(x_2,...,x_n)$ containing $P$ and
$$P\in \mathcal N_m(\mathcal Z),\ m\in\mathcal Z\setminus\mathcal B_{\mathcal Z}\quad\Leftrightarrow\quad m\in\mathcal H,\ \tilde m\in\tilde{\mathcal Z}\setminus
   \mathcal B_{\tilde{\mathcal Z}},\ \tilde P\in\mathcal N_{\tilde m}(\tilde{\mathcal Z}).$$
\end{itemize}
Hence $c_\nu(\mathcal Z)=c_{\nu}(\tilde{\mathcal Z})$
\end{proof}
Let $\tilde F\in\mathbb C[u_1,...,u_n]$ be a homogeneous polynomial
of the form $\tilde F(u_1,...,u_n)=G(u_1^2,...,u_n)$ for some $G\in\mathbb C[u_1,...,u_n]$.
Let $\tilde{\mathcal Z}:=V(\tilde F)\subset\mathbb P^{n-1}$.

We call {\bf algebraic hypersurface of revolution of $\tilde{\mathcal Z}$ around 
the subspace $V(x_1,x_2)$} the hypersurface $\mathcal Z=V(F)\subset\mathbb P^n$ with
$F(x_1,....,x_{n+1}):=G(x_1^2+x_2^2,x_3,...,x_{n+1})$.

Note that if $m[x_1^{(1)}:\cdots:x_{n+1}^{(1)}]\in\mathcal Z\setminus\mathcal H^\infty$ with $x_{n+1}^{(1)}=1$, 
then the "circle" $V(x_1^2+x_2^2-(x_1^{(1)})^2-(x_2^{(1)})^2)\cap\bigcap_{i=3}^{n}V(x_i-x_i^{(1)}x_{n+1})$ of center $[0:0:x_3^{(1)}:\cdots:x_{n+1}^{(1)}]$
that passes through $m$ is contained in $\mathcal Z$.
\begin{prop}
Let $n\ge 3$ and $d\ge 2$.
Let $\mathcal Z=V(F)\subset\mathbb P^n$ be the algebraic hypersurface of revolution of $\tilde{\mathcal Z}=V(\tilde F)\subset \mathbb P^{n-1}$ (with $\tilde F\in\mathbb C[u_1,...,u_n]$ as above) around the subspace $V(x_1,x_2)$, then
$c_\nu(\mathcal Z)=c_{\nu}(\tilde{\mathcal Z})$.
\end{prop}
\begin{proof}

Let $m[x_1^{(1)}:\cdots:x_{n+1}^{(1)}]\in\mathcal Z$
and $P[x_1^{(0)}:\cdots:x_{n+1}^{(0)}]\in\mathbb P^n$.
Then
$$n_{\mathcal Z}(m)[2x_1^{(1)}G_{u_1}({\mathbf m}_1):2x_2^{(1)}G_{u_1}({\mathbf m}_1):
G_{u_2}({\mathbf m}_1):\cdots:G_{u_{n-1}}(\tilde{\mathbf m}_1):0]\in\mathbb P^n,$$
with ${\mathbf m}_1((x_1^{(1)})^2+(x_2^{(1)})^2,x_3^{(1)},...,x_{n+1}^{(1)})\in\mathbb C^n$.
Hence if $m\in\mathcal Z\cap V(x_1^2+x_2^2)\setminus\mathcal B_{\mathcal Z}$,
then $\mathcal N_m\mathcal Z\subset V(x_1^2+x_2^2)$.
Assume from now on that $m\in\mathcal Z\setminus V(x_1^2+x_2^2)$
and that $P\in\mathbb P^n\setminus (V(x_1^2+x_2^2)\cup V(x_1))$.

Let $\tilde m[y_1^{(1)}:x_3^{(1)}:\cdots:x_{n+1}^{(1)}]\in\mathbb P^{n-1}$ 
and $\tilde P[y_1^{(0)}:x_3^{(0)}:\cdots:x_{n+1}^{(0)}]\in\mathbb P^{n-1}$
with $(y_1^{(i)})^2=(x_1^{(i)})^2+(x_2^{(i)})^2$. Note that $\tilde m\in\tilde{\mathcal Z}$.
Then
$$n_{\mathcal Z}(m)[x_1^{(1)}\tilde F_{u_1}(\tilde{\mathbf m})/y_1^{(1)}:x_2^{(1)}\tilde F_{u_1}(\tilde{\mathbf m})/y_1^{(1)}:
\tilde F_{u_2}(\tilde{\mathbf m}):\cdots:\tilde F_{u_{n-1}}(\tilde{\mathbf m}):0]\in\mathbb P^n.$$
\begin{itemize}
\item Note that $m\in\mathcal B_{\mathcal Z}\ \Leftrightarrow\ \tilde m\in\mathcal B_{\tilde{\mathcal Z}}$ (since $x_1^{(1)}$ and  $x_1^{(1)}$ are not both null).
\item Let $\mathcal H=V(\alpha x_1+\beta x_{2})\subset \mathbb P^n$ be a hyperplane that contains $V(x_1,x_2)$ but not contained in $V(x_1^2+x_2^2)$
(i.e. $\alpha^2+\beta^2\ne 0$).
If $m\in\mathcal H\cap\mathcal Z\setminus\mathcal B_{\mathcal Z}$, then
$\mathcal N_m\mathcal Z\subset \mathcal H$.
\item Let $\mathcal H:=V(x_1^{(0)}x_{2}-x_{2}^{(0)}x_{1})$ be the unique
hyperplane that contains $V(x_1,x_2)$ and $P$. Then
$$P\in \mathcal N_m(\mathcal Z),\ m\in\mathcal Z\setminus\mathcal B_{\mathcal Z}\quad\Leftrightarrow\quad m\in\mathcal H,\ \tilde m\in\tilde{\mathcal Z}\setminus
   \mathcal B_{\tilde{\mathcal Z}},\ \tilde P\in\mathcal N_{\tilde m}(\tilde{\mathcal Z}),$$
by choosing $y_1^{(1)}:=y_1^{(0)}x_1^{(1)}/x_1^{(0)}$.
\end{itemize}
Hence $c_\nu(\mathcal Z)=c_{\nu}(\tilde{\mathcal Z})$
\end{proof}
\section{Projective orthogonality in $\mathbb P^n$}\label{NORMAL}
\subsection{From affine orthogonality to projective orthogonality}
Let $E_n$ be an euclidean affine $n$-space of direction the $n$-vector space $\mathbf E_n$
(endowed with some fix basis). Let $\mathbf V:=(\mathbf E_n\oplus \mathbb R)\otimes \mathbb C$
(endowed with the induced basis $\mathbf{e}_1,...,\mathbf{e}_{n+1}$). We consider the complex projective space
$\mathbb P^n:=\mathbb P(\mathbf V)$ with projective coordinates $x_1,...,x_{n+1}$.
Let us write $\pi:\mathbf V\setminus\{0\}\rightarrow\mathbb P^3$ for the canonical projection.
We denote by $\mathcal H^\infty:=V(x_{n+1})\subset \mathbb P^n$ the hyperplane at infinity.
We consider the affine space ${A}^{n}:=\mathbb{P}^{n}\setminus \mathcal{H}^{\infty }$ endowed 
with the vector space $\overrightarrow{\mathbf E}:=Span(\mathbf e_1,\cdots,\mathbf e_n)
\subset \mathbf V$ (with the affine structure 
$m+\overrightarrow{\mathbf v}=\pi(\mathbf m+\overrightarrow{\mathbf v})$ if $\overrightarrow{\mathbf v}\in \overrightarrow{\mathbf E}$ 
and $m=\pi(\mathbf m)\in A^n$ with
$\mathbf m(x_1,\cdots,x_n,1)$).

Let us consider $\mathcal{W}_1= \mathbb{P}(\mathbf W_1)
\subset\mathbb P^n$ 
and $\mathcal{W}_2 =\mathbb{P}(\mathbf W_2)\subset\mathbb P^n$ where 
$\mathbf W_1$ and $\mathbf W_2$ are two vector subspaces of $\mathbf V$
 not contained in $\overrightarrow{\mathbf E}$ 
such that $\dim \mathbf W_1+\dim\mathbf W_2=n+2$. 
Since $\mathcal W_i$ is not contained in $\mathcal H^\infty$,
$ W_i:=\mathcal{W}_i\setminus\mathcal H^\infty$ is 
an affine subspace of $A^n$ with vector space
$\overrightarrow{\mathbf{W}_i}:=\mathbf W_i \cap\overrightarrow{\mathbf E}$, 
that is to say that 
there exists $m_i$ such that $W_i=m_i+\overrightarrow{\mathbf{W}_i}$ in $ A^n$.
Consider the usual bilinear symmetric form $\langle u,v\rangle 
=\sum_{i=0}^3u_{i}v_{i}$ on $\mathbf V$, the associated orthogonality on $\mathbf V$ is
written $\boldsymbol{\perp} $.
\begin{defi}
Let us consider $\mathcal{W}_1= \mathbb{P}(\mathbf W_1)
\subset\mathbb P^n$ 
and $\mathcal{W}_2 =\mathbb{P}(\mathbf W_2)\subset\mathbb P^n$ where 
$\mathbf W_1$ and $\mathbf W_2$ are two vector subspaces of $\mathbf V$
 not contained in $\overrightarrow{\mathbf E}$ and
such that $\dim \mathbf W_1+\dim\mathbf W_2=n+2$. 
With the above notations, we say that $\mathcal W_1$ and $\mathcal W_2$ are orthogonal in $\mathbb{P}^{3}$ if $\overrightarrow {\mathbf W}_1\perp \overrightarrow{\mathbf{W}}_2$. 
We then write ${\mathcal W_1\boldsymbol{\perp }\mathcal W_2}$.
\end{defi}
Note that if $\mathcal H\subset \mathbb P^n$ and $\mathcal L
\subset\mathbb P^n$ are respectively an hyperplane and a line in
$\mathbb P^n$ not contained in $\mathcal H^\infty$, then
$\mathcal H\perp\mathcal L$ if and only if the point at infinity of $\mathcal L$
is the pole in $\mathcal H^\infty$ of the line $\mathcal H\cap\mathcal H^\infty\subset\mathcal H^\infty$ with respect to the {\bf umbilical}
$\mathcal U_\infty:=V(x_1^2+...+x_n^2)\cap\mathcal H^\infty\subset\mathcal H^\infty$. 
This leads us to the following generalization of normal lines to an hyperplane.
\begin{defi}
We say that a projective hyperplane $\mathcal H=V(a_1x_1+\cdots+a_{n+1}x_{n+1})\subset\mathbb P^n$ and a projective line $\mathcal L=\mathbb P(\mathbf L)\subset\mathbb P^n$ are orthogonal in $\mathbb P^n$ if 
$(a_1,\cdots,a_n,0)\in\mathbf L$. We then write $\mathcal L\perp\mathcal H$.
\end{defi}
It is worthful to note that, with this definition, an orthogonal line to 
an hyperplane $\mathcal H$ may be included in $\mathcal H$.

\end{appendix}

\end{document}